\documentclass{mcom-l}  

\usepackage[dvips]{graphicx}
\usepackage{amsmath}
\usepackage{amsfonts}

\newtheorem{theorem}{Theorem}[section]
\newtheorem{lemma}[theorem]{Lemma}
\newtheorem{proposition}[theorem]{Proposition}

\theoremstyle{definition}

\theoremstyle{remark}

\numberwithin{equation}{section}

\newcommand \im {\text{Im}} 
\newcommand \rank {\text{rank}}  

\newcommand \eps \varepsilon

\newcommand \be     {\begin{equation}}
\newcommand \ee     {\end{equation}}
\newcommand \ben    {\begin{eqnarray}}
\newcommand \een    {\end{eqnarray}}

\newcommand \RR     {\mathbb{R}}
\newcommand \ZZ     {\mathbb{Z}}

\newcommand \dt     {\partial_t}
\newcommand \e      {\varepsilon}
\newcommand \dx     {\partial_x}

\newcommand \demi   {{1/2}}

\newcommand \Dcal   {{\mathcal D}}
\newcommand \Ecal   {{\mathcal E}}
\newcommand \Lcal   {{\mathcal L}}
\newcommand \Mcal   {{\mathcal M}}
\newcommand \Ncal   {{\mathcal N}}
\newcommand \Ocal   {{\mathcal O}}
\newcommand \Rcal   {{\mathcal R}}

\begin{document}
\bibliographystyle{plain}
\title[ Late-time/stiff relaxation asymptotic-preserving approximations]
{Late-time/stiff relaxation asymptotic-preserving approximations of hyperbolic equations}


\author[C. Berthon]{Christophe Berthon}
\address{Universit\'e de Nantes, Laboratoire de Math\'ematiques Jean Leray, CNRS UMR 6629,
2 rue de la Houssini\`ere, BP 92208, 44322 Nantes, France.}
\email{Christophe.Berthon@math.univ-nantes.fr}

\author[P.G. L{\tiny e}Floch]{Philippe G. L{\smaller e}Floch}
\address{Laboratoire Jacques-Louis Lions, Centre National de la Recherche Scientifique, 
Universit\'e Pierre et Marie Curie (Paris 6), 4 Place Jussieu,
75252 Paris, France.}
\email{pgLeFloch@gmail.com}

\author[R. Turpault]{Rodolphe Turpault}
\address{Universit\'e de Nantes\\Laboratoire de Math\'ematiques Jean Leray, CNRS UMR 6629, 
2 rue de la Houssini\`ere, BP 92208, 44322 Nantes, France.}
\email{Rodolphe.Turpault@univ-nantes.fr}



\begin{abstract}
We investigate the late-time asymptotic behavior of solutions to 
nonlinear hyperbolic systems of conservation laws containing stiff relaxation terms. 
First, we introduce a Chapman-Enskog-type asymptotic expansion and derive an
effective system of equations describing the late-time/stiff relaxation singular limit.
The structure of this new system is discussed and the role of 
a mathematical entropy is emphasized. Second, 
we propose a new finite volume discretization which, in late-time asymptotics, 
allows us to recover a discrete version of the same effective asymptotic system. This is 
achieved provided we suitably discretize the relaxation term in a way that 
depends on a matrix-valued free-parameter, chosen so that the desired asymptotic behavior is 
obtained. Our results are illustrated with several models of interest in 
continuum physics,
and numerical experiments demonstrate the relevance of the proposed theory and numerical strategy.
\end{abstract}

\maketitle



\section{Introduction}

\subsection*{Motivations}

We are interested in the numerical approximation of solutions
to nonlinear hyperbolic systems of conservation laws containing stiff relaxation terms. An extensive literature is available 
on such systems since they arise in many physical problems of interest, 
for instance, in the modeling of complex multiphase flows involving phase
transitions or kinetic-type phenomena. (See, for instance, 
\cite{bib_cpe,bib_crispel,bib_afeintou,bib_jx,bib_suliciu}.)  
Stiff relaxation source terms are essential to model phenomena involving distinct physical time-scales. 
The derivation and the analysis (existence, stability) of an effective system of equations 
(also of hyperbolic type as the original system) 
as the relaxation time
goes to zero is required for understanding the behavior of general solutions.
(See, for instance, \cite{bib_liu,bib_suliciu2,bib_yong,bib_yz}.) 

In the present paper, we go beyond these classical works and investigate the {\bf late-time} behavior 
of solutions to systems with stiff relaxation and, specifically,
we consider the class of systems ($\eps>0$) 
\be\label{syst_1}
\eps \, \dt U +\dx F(U) =-\frac{R(U)}{\eps},
\qquad t>0,~ x\in\RR,
\ee
where the main unknown $U:\RR\times\RR^+\to\Omega$ 
takes its values in a convex set $\Omega\subset\RR^N$. The associated
homogeneous first-order system 
---obtained by neglecting the relaxation term $R: \Omega \to \RR^N$ in the right-hand side of
\eqref{syst_1}---
is assumed to be hyperbolic in the sense that the $N\times N$ matrix $A:=D_U F$ admits real 
eigenvalues and  a full basis of eigenvectors.

In comparison with classical works on this subject 
\cite{bib_cll} as well as also \cite{bib_an,bib_bian,bib_chalons,bib_liu}, 
the main novelty in the present work lies in the fact that 
the term $\eps \, \dt U$ is rescaled to be proportional to $\eps$. 
When $\eps$ tends to zero, this will lead us to limit solutions whose behavior is quite distinct from the ones 
of standard relaxation limits. Indeed, as we establish below, the effective problem associated with \eqref{syst_1}  
turns out to be a {\sl diffusion problem,} 
in which the diffusion operator is determined by the (nonlinear) relaxation term $R$, in a way 
explicitly determined in the present work\footnote{
After this paper was completed, P. Marcati pointed out to the authors the relevant references \cite{DM,MR}
in  which several convergence results are established.  }. 
The relaxation map $R:\Omega \to \RR^N$ is assumed to be sufficiently regular and satisfy the 
conditions introduced by Chen, Levermore, and Liu \cite{bib_cll} for the ``hyperbolic to hyperbolic'' relaxation problem.

First of all, we assume the existence of a (constant)  
$n\times N$ matrix $Q$ with (maximal) rank $n<N$ such that
\be
\label{H1_1}
QR(U)=0, \qquad U\in\Omega.
\ee
From this, it follows that, if $U$ is a solution to \eqref{syst_1}, then 
$QU: \RR\times\RR^+\to \RR^N$ satisfies the system of $n$ conservation laws 
\be\label{H2_1}
\eps \, \dt (QU) +\dx (QF(U)) = 0,
\ee
in which the unknown takes values in the (convex) set 
$$
\omega := Q\Omega\subset \RR^n.
$$
We also assume that a map $\Ecal:\omega\to\Omega$ uniquely determines local equilibria, as defined by the relations
\be\label{H3_1}
Q \, \Ecal(u)=u, 
\quad
R(\Ecal(u))=0,
\qquad u\in\omega.
\ee
This suggests to introduce the {\sl equilibrium manifold}
$$
\Mcal:=\big\{ U=\Ecal(u) \big\}.  
$$
The dimension of the null space of the $N\times N$ matrix $B:=DR$ with the equilibrium submanifold
is assumed to be ``maximal'' in the sense that  
\begin{align}
\label{kerDR}
\dim\Big( \ker(B(\Ecal(u)))\Big)&=n,
\\
\label{kerImDR}
\ker\big( B(\Ecal(u)) \big) \cap \im\big( B(\Ecal(u))\big) & =\{0\}.
\end{align}
Finally, since we are interested in the late-time behavior of solutions it is necessary to impose that
$$
QF(\Ecal(u))= c, \qquad  u\in\omega,
$$
for some constant vector $c \in \RR^n$.  Indeed, in view of (\ref{H2_1}),
we can formally write
$$
\eps \, \dt u + \dx QF(\Ecal(u))  \to 0,
$$
so that $QF(\Ecal(u))$ must be a constant, normalized so that 
\be
\label{H4_1}
QF =0 \qquad \text{ on } \Mcal. 
\ee


\subsection*{Main results}

Our objective in this paper is two-fold.  First, we determine an effective system of equations driving 
the late-time asymptotic behavior of general solutions to (\ref{syst_1}) in the singular regime $\eps \to 0$. 
Second, we introduce a novel numerical strategy allowing us to
precisely  recover the expected asymptotic behavior.  

The paper is organized as follows. Section~2 is devoted to the asymptotic analysis: 
we derive an effective system and, under the hypotheses stated in this introduction as well as 
the assumption of the existence of a mathematical entropy compatible with the relaxation
source term (in a sense specified below), 
we study the structure of this system which is found to be of diffusion-type 
and we show that the associated total entropy is non-increasing in time. 

Then, in Section 3 below, we generalize our analysis to a nonlinear version of (\ref{syst_1}), since 
this is required to encompass certain models arising in the aplications when the relaxation 
has strong nonlinearities. The class of systems under consideration in this section reads  
\be\label{systext_1}
\eps \, \dt U+\dx F(U) = -\frac{R(U)}{\eps^m},
\ee
where the parameter $m\ge 1$ introduces an {\sl additional scale} in the problem.  
By imposing that 
$$
R\big(\Ecal(u)+\eps \, U\big) = \eps^m \, R\big(\Ecal(u)+M \, U\big),
\qquad U\in\Omega, \quad u\in\omega,
$$
plus certain conditions on the $N\times N$ matrix $M$ (see Section~3), 
we derive again an effective system which, now, consists of  
{\sl nonlinear diffusion equations.} 

Next, in Section~4, we discuss several specific examples of practical interest and show that 
they fit into our general framework:  (1) the Euler equations of compressible fluids with friction term, 
(2) the so-called $M1$-model of radiative transfer, 
(3) a new model that couples the former two models
and, finally, (4) the shallow-water system with nonlinear friction. Interestingly, the latter 
system does require the more general formalism (\ref{systext_1}).

In Section~5, we turn our attention to the discretization of  (\ref{syst_1})  and (\ref{systext_1}).
Recall that the ``hyperbolic to hyperbolic" relaxation problem was treated numerically
first by Jin and Xin \cite{bib_jx}; our objective is to treat here 
the ``hyperbolic to parabolic'' relaxation
problem. 
We introduce a new Godunov-type finite volume scheme which incorporates a suitable discretization of the source term 
and allows us to recover the expected asymptotic regime. The discrete form of the source term
is derived by modifying a Riemann solver associated with  
the homogeneous system. The proposed numerical strategy is proven to satisfy a domain invariant principle
in the sense that all of the computed states belong to the convex set $\Omega$. 
Finally, in Section~6, we conclude with several numerical experiments with the proposed asymptotic-preserving schemes,
and we demonstrate the interest of our method on several of the physical models.


\section{Late-time/stiff-relaxation framework}

\subsection*{Derivation of the effective system}

Throughout this section, we consider the nonlinear system of balance laws \eqref{syst_1}, and 
we impose the conditions stated in the introduction. 

Our objective is to exhibit the system of effective equations satisfied by local equilibria $u= u(t,x) \in \omega$. 
In the spirit of Chapman-Enskog expansions of the kinetic theory \cite{bib_cer},  
we consider a formal expansion of solutions $U$ to (\ref{syst_1})
in the form:
\be\label{expansion}
U^\eps = \Ecal(u) + \eps \, U_1 +\eps^2 \, U_2 + \ldots, 
\ee
where $U_i$ is referred to as the ${\rm i^{th}}$-order corrector. Such an expansion is natural 
in view of the assumptions (\ref{H3_1})
and (\ref{H4_1}).  We consider the system 
\be\label{eqexpan}
\begin{aligned}
&\eps \, \dt U^{\eps} + \dx F(U^\eps) =
   -\frac{1}{\eps} \, R(U^\eps),\\
&QU^\eps = u,
\end{aligned}
\ee
in which we plug the formal expansion (\ref{expansion}) and 
match together terms of the same order of magnitude in $\eps$.

From $QU^\eps=u$ we deduce that  $Q\Ecal(u)=u$, and 
$QU_i=0$ for
all ${\rm i^{th}}$-order correctors. Next, for sufficiently regular flux $F$, we 
obtain the following expansion:
\be\label{dlF}
\begin{aligned}
F(U^\eps) =& F(\Ecal(u)) + \eps A(\Ecal(u)) \, U_1
+
   \frac{\eps^2}{2}D^2_U F(\Ecal(u)) \, (U_1,U_1)  \\
   & +
   \eps^2 A(\Ecal(u)) \, U_2 + \Ocal(\eps^3).
\end{aligned}
\ee
Analogously, it is useful to write down the formal expansion of the
relaxation term. Taking (\ref{H3_1}) into account, we obtain
\be\label{dlR}
\frac{1}{\eps}R(U^\eps)=B(\Ecal(u)) \, U_1 +
   \frac{\eps}{2}D^2_U R(\Ecal(u)).(U_1,U_1) +
   \eps B(\Ecal(u)) \, U_2 + \Ocal(\eps^2).
\ee
We then plug (\ref{dlF}) and (\ref{dlR}) into (\ref{eqexpan}) and obtain
\be\label{eqdl}
\begin{aligned}
& \eps\dt\Ecal(u) + \dx F(\Ecal(u)) + \eps\dx A(\Ecal(u)) \, U_1
\\ &
= -B(\Ecal(u)) \, U_1 -\frac{\eps}{2} D^2_U R(\Ecal(u)).(U_1,U_1)
-\eps B(\Ecal(u)) \, U_2 +\Ocal(\eps^2).
\end{aligned}
\ee

Let us consider \eqref{eqdl} and expand in $\eps$. The
zeroth-order terms give 
$$
\dx F(\Ecal(u)) = -B(\Ecal(u)) \, U_1.
$$
At this stage, we can solve for $U_1$ by recalling that  $QU_1 =0$. Indeed,
since we have imposed the properties (\ref{kerDR}) and (\ref{kerImDR}) on the null space of $B(\Ecal(u))$ 
then, for all fixed $J\in\RR^N$, the system
\be\label{eqlin}
\begin{aligned}
& B(\Ecal(u)). V =J,\\
& QV =0,
\end{aligned}
\ee
admits a unique solution $V\in\RR^n$ if and only if $QJ=0$.
(See Lemma~\ref{lemanx} at the end of this section.) Here, we
have $J= -\dx F(\Ecal(u))$ which satisfies $Q\dx F(\Ecal(u))=0$ by
(\ref{H4_1}). As a consequence, for all $u\in\omega$, we can uniquely
determine $U_1$ such that 
\be\label{eqU1}
\begin{aligned}
& B(\Ecal(u)). U_1 = -\dx F(\Ecal(u)),\\
& QU_1 =0.
\end{aligned}
\ee

Next, we consider the first-order terms in (\ref{eqdl}), and obtain 
$$
\dt\Ecal(u)+\dx A(\Ecal(u)) \, U_1 = -\frac{1}{2} D^2_U
R(\Ecal(u)).(U_1,U_1)- B(\Ecal(u)) \, U_2.
$$
Multiplying by $Q$ and using the assumption (\ref{H1_1}), we
obtain
$$
 QD^2_U R(\Ecal(u)).(U_1,U_1) \equiv 0, \qquad
 QB(\Ecal(u)) \, U_2 \equiv 0.
$$
Since $Q\Ecal(u)=u$, we arrive at an effective system for the limit $u$, that is, 
\be\label{eqlim}
\dt u = -\dx \left( QA(\Ecal(u)) \, U_1 \right),
\ee
where $U_1$ is the unique solution to \eqref{eqU1}.


\subsection*{The role of a mathematical entropy}

Now, we assume the existence of a sufficiently regular, 
 mathematical entropy 
$\Phi:\Omega\to\RR$ so that the matrix $D^2_U \Phi(U)A(U)$ is
symmetric for all $U$ in $\Omega$ and there exists an entropy-flux map
$\Psi:\Omega\to\RR$ such that
\be\label{defpsi}
D_U \Phi(U) A(U) = D_U \Psi(U),\qquad U\in\Omega.
\ee
Hence, all smooth solutions to (\ref{eqexpan}) satisfy
\be\label{entropexpan}
\eps\dt\Phi(U^\eps) +\dx\Psi(U^\eps) = 
-\frac{1}{\eps} D_U \Phi(U^\eps) R(U^\eps).
\ee
As usual, we impose $\Phi$ to be convex by requiring the $N\times N$
matrix $D^2_U \Phi(U)$ to be positive. In addition, we assume that the
entropy is compatible with the relaxation in the sense that, for  
some map $\nu: \Mcal \to \RR^n$,  
\ben
&& \label{cond1}
  D_U \Phi(U) R(U) \ge 0, \qquad U\in\Omega,`
\\
&& \label{cond2}
D_U \Phi = \nu \, Q \qquad \text{ on } \Mcal. 
\een

We now analyze the nature of the limiting system (\ref{eqlim}).

\begin{theorem}\label{maintheo}
Consider the nonlinear system of balance laws \eqref{syst_1}, under 
the assumptions \eqref{H1_1}--\eqref{H4_1} and \eqref{cond1}--\eqref{cond2}. 
Then, the associated limiting system (\ref{eqlim}) takes the form
\be\label{eqlim2}
\dt u = \dx \left(
S\Lcal^{-1}(u)S^T\left(\dx D_u \Phi(\Ecal(u))\right)^T
\right),
\ee
where
\be\label{defS}
S := Q A(\Ecal(u))
\ee
and
\be\label{defLu}
\Lcal(u) := D^2_U\Phi(\Ecal(u)) B(\Ecal(u)),
\ee
and, for all $b\in\RR^N$ with $Qb =0$, $\Lcal^{-1}(u).b$ denotes the
unique solution to the system
$$
\Lcal(u).V =b, \quad \quad QV =0.
$$
In addition, this system is dissipative with respect to the
entropy $\Phi$ in the sense that 
 the following positivity condition holds for all $u$ in $\omega$: 
$$
\big( \dx D_u \Phi(\Ecal(u))\big) \, S\Lcal^{-1}(u)S^T
\, \big( \dx D_u
\Phi(\Ecal(u)) \big)^T \ge 0.
$$
\end{theorem}

\begin{proof} We follow the strategy in \cite{bib_cll} which we adapt to our problem. 
We first establish (\ref{eqlim2}). To simplify the notation, we set
$$
\Dcal(u) = -QA(\Ecal(u)) \, U_1,
$$
to restate (\ref{eqlim}) as follows:
$$
\dt u= \dx \Dcal(u).
$$
Here, the vector $U_1\in\RR^N$ is the unique solution of
(\ref{eqU1}). Since $D^2_U \Phi(U)$ is a positive $N\times N$ matrix
for all $U\in\Omega$, we can rewrite (\ref{eqU1}) as follows:
\be\label{eqU1aux}
\begin{aligned}
& D^2_U \Phi(\Ecal(u))B(\Ecal(u)). U_1 = 
  -D^2_U \Phi(\Ecal(u)) \dx F(\Ecal(u)),\\
& Q U_1 =0.
\end{aligned}
\ee
Invoking (\ref{defLu}), the state vector $U_1$ is defined as the
unique solution of
$$
\begin{aligned}
& \Lcal(u) \, U_1 = -D^2_U \Phi(\Ecal(u)) \dx F(\Ecal(u)),\\
& Q U_1 =0.
\end{aligned}
$$
As a consequence, by definition of $\Lcal^{-1}(u)$ we have
$$
U_1 = -\Lcal^{-1}(u)D^2_U \Phi(\Ecal(u)) \dx F(\Ecal(u)).
$$
From (\ref{defS}), we find 
$$
\Dcal(u) = S \Lcal^{-1}(u)D^2_U \Phi(\Ecal(u)) \dx F(\Ecal(u)).
$$
To obtain (\ref{eqlim2}), we then establish
\be\label{aux0}
D^2_U \Phi(\Ecal(u)) \dx F(\Ecal(u)) = S^T v,
\ee
where we define $v\in\RR^n$ as follows:
\be\label{defv}
v = \left( \dx D_u \Phi(\Ecal(u)) \right)^T.
\ee
To this end we differentiate (\ref{cond2}) (which involves the map $\nu$) and we obtain
$$
\left(D_u \Ecal(u) \right)^T D^2_U \Phi(\Ecal(u))
=
D^2_u\Phi(\Ecal(u))Q.
$$
By transposition, we thus get
\be\label{aux1}
D^2_U \Phi(\Ecal(u)) D_u \Ecal(u) = Q^T D^2_U \Phi(\Ecal(u)).
\ee
Now, we have
$$
D^2_U\Phi(\Ecal(u)) \dx F(\Ecal(u)) =
D^2_U \Phi(\Ecal(u)) A(\Ecal(u)) D_u \Ecal(u) \dx u
$$
Since the matrix $D^2_U\Phi(\Ecal(u)) A(\Ecal(u))$ is symmetric, we can
write
$$
D^2_U\Phi(\Ecal(u)) \dx F(\Ecal(u)) =
\left(A(\Ecal(u))\right)^T D^2_U \Phi(\Ecal(u)) D_u \Ecal(u) \dx u.
$$
We use (\ref{aux1}) and obtain
\be
\aligned
D^2_U\Phi(\Ecal(u)) \dx F(\Ecal(u))
  &= \left(A(\Ecal(u))\right)^T Q^T D^2_u \Phi(\Ecal(u)) \dx u
\\
&= S^T v, \label{eqSTv}
\endaligned
\ee 
and the identity (\ref{eqlim2}) follows.

The proof will be completed as soon as we establish
$$
v^TS \Lcal^{-1}(u) S^T v \ge 0,
$$
which is equivalent to $v^T \Dcal(u) \ge 0$ where $\Dcal(u) = -S U_1$. We thus
have
$$
v^T\Dcal(u)=-v^T S U_1.
$$
But from (\ref{eqU1aux}) and (\ref{aux0}), we deduce 
$$
v^T\Dcal(u) = \left(
D^2_U \Phi(\Ecal(u)) B(\Ecal(u)) U_1
\right)^T U_1.
$$
As a consequence, the expected inequality $v^T \Dcal(u) \ge 0$ will hold as soon as we prove that the matrix $D^2_U \Phi(\Ecal(u)) D_U
R(\Ecal(u))$ is non-negative.

Recalling the entropy assumption (\ref{cond1}) and the equilibrium
property $R(\Ecal(u))=0$, we have
\ben
&& \nonumber
  D_U \Phi(U) R(U) \ge 0,\qquad  U\in\Omega,\\
&& \nonumber
  \left(D_U \Phi(U) R(U)\right)|_{ U=\Ecal(u)}=0,\qquad  u\in\omega.
\een
As a consequence, the matrix $D^2_U\left(D_U \Phi(U)
R(U)\right)|_{ U=\Ecal(u)}$ is non-negative. Next, a calculation using the chain rule 
and the fact that $R$ vanishes on the equilibrium submanifold, that is, 
$$
R(\Ecal(u)) = 0, 
$$
leads us easily to 
$$
D^2_U\left(D_U \Phi(U) R(U)\right)|_{ U=\Ecal(u)} =
  D^2_U\Phi(\Ecal(u)) B(\Ecal(u)) +
\left( D^2_U\Phi(\Ecal(u)) B(\Ecal(u)) \right)^T,
$$
in which no third-order derivative terms arise since $R$ precisely vanishes on the equilibrium manifold. 
From the above identity, it then follows that 
\be\label{positivity}
\left( (D^2_U\Phi(\Ecal(u)) B(\Ecal(u))) \, U \right)^T U \ge 0,
\qquad U\in\Omega.
\ee
We thus obtain the expected inequality $v^T \Dcal(u) \ge 0$ and the
proof is completed.
\end{proof}


\subsection*{Monotonicity of the entropy}

We then study the asymptotic behavior of the entropy
inequality. In order to exhibit the entropy law satisfied by the
equilibrium solution $\Ecal(u)$, we need the following technical result.

\begin{lemma}\label{lpsicons}
Under the assumptions of Theorem~\ref{maintheo}, 
the entropy-flux map restricted to the equilibrium
$\Psi(\Ecal(u))$ remains constant for all $u$ in $\omega$.
\end{lemma}

\begin{proof}We consider the map $u\mapsto\Psi(\Ecal(u))$ and, after differentiation, obtain 
$$
D_u \Psi(\Ecal(u)) = D_U \Psi(\Ecal(u)) D_u \Ecal(u).
$$
By definition of $\Psi$ given by (\ref{defpsi}), we have 
$$
D_u \Psi(\Ecal(u)) = D_U \Phi(\Ecal(u)) A(\Ecal(u)) D_u \Ecal(u).
$$
Then, the assumption (\ref{cond2}) made on $\Phi$ yields the
following relation:
\be
\aligned
D_u \Psi(\Ecal(u)) &= D_u \Phi(\Ecal(u)) Q A(\Ecal(u)) D_u
\Ecal(u), 
\\
&=  D_u \Phi(\Ecal(u)) D_u QF(\Ecal(u)). 
\endaligned
\ee 
Since we have $QF(\Ecal(u))=0$ over $\omega$, then $D_u QF(\Ecal(u))
=0$. As a consequence, $D_u\Psi(\Ecal(u))=0$ for all $u$ in $\omega$
and the proof is completed.
\end{proof}

Equipped with this result, we can exhibit the asymptotic equation
satisfied by the equilibrium entropy $\Phi(\Ecal(u))$. Arguing the
formal asymptotic expansion (\ref{expansion}) satisfied by $U$:
$$
U^\eps = \Ecal(u) + \eps U_1 + \ldots,
$$
where $U_1$ is defined as the unique solution to (\ref{eqU1}), we
consider the formal expansion of each term in (\ref{entropexpan}). First,
since the entropy and entropy-flux are regular maps, we have
\be\label{exp1}
\Phi(U^\eps)=\Phi(\Ecal(u)) +\eps D_U\Phi(\Ecal(u)) \, U_1 + \Ocal(\eps^2),
\ee
and
\be
\Psi(U^\eps)=\Psi(\Ecal(u)) +\eps D_U\Psi(\Ecal(u)) \, U_1 + \Ocal(\eps^2).
\ee
But, by applying Lemma \ref{lpsicons}, we deduce the following
relation:
\be\label{exp2}
\dx \Psi(U^\eps) = \eps\dx D_U\Psi(\Ecal(u)) \, U_1 + \Ocal(\eps^2).
\ee
Similarly, concerning the entropy relaxation source term, we easily
have:
\be\label{exp3}
D_U \Phi(U^\eps)R(U^\eps) = \eps^2D^2_U \Phi(\Ecal(u)) D_U
R(\Ecal(u)) U_1 + \Ocal(\eps^3).
\ee
Now, we plug the expressions (\ref{exp1}), (\ref{exp2}) and
(\ref{exp3}) into (\ref{entropexpan}), and consider the first-order
terms only: 
\be\label{entropequi}
\dt\Phi(\Ecal(u)) = -\dx \left( D_U \Psi(\Ecal(u)) \, U_1\right)
- U_1^T \left( D^2_U \Phi(\Ecal(u)) B(\Ecal(u)) \right)  U_1,
\ee
where, once again, $U_1$ is the unique solution to (\ref{eqU1}). From
this entropy evolution law, we can state an entropy decreasing
principle. Indeed, we have established that the matrix $D^2_U \Phi(\Ecal(u)) D_U
R(\Ecal(u))$ is positive (cf.~(\ref{positivity})
and, as a consequence, 
$$
U^T \left( D^2_U \Phi(\Ecal(u)) B(\Ecal(u)) \right) U \ge 0,
\qquad  U\in\Omega.
$$
We have thus proven the following statement. 

\begin{proposition}
\label{lentrop}
Under the assumptions of Theorem~\ref{maintheo}, the entropy is non-increasing in the following sense:
$$
\dt\Phi(\Ecal(u)) \le -\dx \left( D_U \Psi(\Ecal(u)) \, U_1\right).
$$
\end{proposition}

To conclude this presentation of the asymptotic system of diffusion equations
satisfied by stiff relaxation term for late times, let us 
emphasize the role played by the entropy. As recognized by 
Chen, Levermore and Liu \cite{bib_cll}, the existence of a convex 
mathematical entropy
provides an important structure to investigate the asymptotic regime
satisfied by the model. However, the main discrepancy with \cite{bib_cll}
lies in the nature of the singular limit system. Indeed, in \cite{bib_cll},
the singular limit system turns out to be an hyperbolic system
supplemented by an $\eps$ first-order diffusive term. Here, the
obtained asymptotic system defines a system of diffusion equations. Even if the
limiting solution is smooth (due to the diffusive
nature of the limiting system), the mathematical entropy is essential in order 
to establish the stability of the asymptotic regime.


\subsection*{A technical lemma}

We end this section with  a technical lemma that was useful in the above derivation of the asymptotic system. 

\begin{lemma}
\label{lemanx}
Let $A$ be a real $N\times N$ matrix such that $\dim(\ker(A))=n$ and $\ker(A)\cap \im(A)=\{0\}$. 
Let also $Q$ be a real $n\times N$ matrix such that $\rank(Q)=n$ and $QA=0$ 
Then, for any $b\in\RR^N$, the linear system 
\begin{equation}
\label{sysrect}
\aligned
Ax & = b,
\\
Qx & = 0,
\endaligned
\end{equation}
admits a unique solution if and only if $Qb=0$.
\end{lemma}

\begin{proof}
If the system (\ref{sysrect}) has a solution then left-multiplying $Ax=b$ by $Q$ leads to $QAx=Qb$. Since $QA=0$ then $Qb=0$.\\
Now we suppose that $Qb=0$. Since $QA=0$ then the columns of $Q^T$ are elements of the left null-space of $A$. Furthermore, $\dim(\ker(A^T)) = n = \rank(Q)$ therefore the columns of $Q^T$ are a basis of the left null-space of $A$. This implies that if $z_0\in \ker(A^T)$ then there exists $y_0\in \RR^n$ such that $z_0=Q^T y_0$.

Let us consider $z\in \RR^N$ such that $Qz=0$. Using the fundamental theorem of linear algebra we 
have $z=z_0+z_1$ where $z_0\in \ker(A^T)$ and $z_1\in \im(A)$.\\
We then characterize $z_0$ and $z_1$ using their respective definitions: there exists $y_1\in\RR^N$ such that $z_1=Ay_1$,
 and
there exists $y_0\in\RR^n$ such that $z_0=Q^Ty_0$. With these characterizations we have:
\begin{align*}
0&=Qz=Qz_0+Qz_1
\\
&=QQ^T y_0+QAy_1 =QQ^Ty_0+0, 
\end{align*}
since $QA=0$. Now, $\rank(Q)=n$ implies that $QQ^T$ is a symmetric positive-definite matrix and therefore that $y_0=0$. Therefore $z\in \im(A)$ if and only if $Qz=0$.

By considering first the existence issue, 
the above property on $b$ allows us to say that $b\in \im(A)$ and hence there exists 
$x\in \RR^N$ such that $Ax=b$. If $x$ is such a solution and since $\ker(A)\oplus \im(A)=\RR^N$, we have
 $x=x_0+x_1$ where $x_0\in \ker(A)$ and $x_1\in \im(A)$. With these notation $x_1$ is a solution to (\ref{sysrect}).\\

Considering next the uniqueness issue and
suppose that $x_1$ and $x_2$ are two solutions to (\ref{sysrect}). Then $(x_1-x_2)$ is the solution to
\begin{equation*}
\aligned
A(x_1-x_2) & = 0,
\\
Q(x_1-x_2) & = 0,
\endaligned 
\end{equation*}
and therefore $(x_1-x_2)\in \ker(A)\cap \im(A)=\{0\}$. Hence, the solution to (\ref{sysrect}) is unique.
\end{proof}


\section{Nonlinear diffusive regime}

\subsection*{Derivation of a nonlinear asymptotic system}

Some physical models involve several relaxation time-scales. 
More precisely, we suppose now that the ratio of the relaxation time and the late time under consideration 
is no longer constant. The extended model thus reads 
\be\label{systnl}
\eps\dt U+\dx F(U) =
-\frac{1}{\eps^m} R(U),
\qquad t>0, \quad x\in\RR,
\ee
with $U\in\Omega\subset\RR^N$. Here, $m\ge 1$ denotes an
integer. The case $m=1$ has been discussed in the previous section, and we now assume $m>1$.

Several of the assumptions made for $m=1$ are kept here. Precisey, we assume
the existence of an $n\times N$ matrix $Q$ with rank $n<N$ satisfying 
(\ref{H1_1}), as well as the existence of a map 
$\Ecal:\omega\to\Omega$ satisfying (\ref{H3_1}). We also impose that 
the flux $F$ satisfies (\ref{H4_1}).
Concerning the nonlinear relaxation map $R$, an additional
assumption must be imposed: there exists a $N\times N$ matrix,
denoted $M(\eps)$, such that
\be\label{Hadd}
R(\Ecal(u)+\eps U) =\eps^m R\big( \Ecal(u) + M(\eps)U \big),
\qquad  U\in\Omega, \quad u\in\omega.
\ee
The matrix $M(\eps)$ is assumed to be sufficiently smooth in $\eps \in [0,1]$. 

The assumption initially made on the
kernel of the matrix $B(\Ecal(u))$ is irrelevant if
$m>1$. Indeed, this kernel assumption was imposed to ensure the existence
and uniqueness of the solution to (\ref{eqlin}). We are going to see
that this linear system is no longer relevant and must be replaced 
by a nonlinear problem. In
this sense, the proposed extension is called {\it nonlinear} since the
diffusive asymptotic regime will involve a nonlinear
differential operator.

First, to derive the effective system of equations satisfied by the
local equilibrium $u\in\omega$, we introduce again a Chapman-Enskog-type expansion: 
$$
U^\eps = \Ecal(u) +\eps U_1 +\eps^2 U_2+...
$$
We plug this expansion into (\ref{systnl}) and match terms of the same order in $\eps$.

Enforcing $QU^\eps=u$, the condition (\ref{H3_1}) on the local
equilibrium implies $QU_i=0$ for each corrector
term. Note that the expansion (\ref{dlF}) for the flux remains valid, 
and we only have to evaluate the expansion of the
relaxation term by recalling (\ref{Hadd}). Indeed, this
assumption gives
$$
R(U^\eps)=\eps^m \, R\big( \Ecal(u) +M(\eps) U_1 +\eps M(\eps) U_2+...
\big),
$$
and so 
\be\label{dlRm}
\begin{aligned}
\frac{1}{\eps^m}R(U^\eps) =&
\, R(\Ecal(u)+M(0)U_1) + \eps B(M(0)U_1).(M(0)U_2)
 \\
 & \, 
+ \eps \, D_\eps (M(\eps)U_1)|_{\eps=0}.
  R(\Ecal(u)+M(0)U_1) + \Ocal(\eps^2).
\end{aligned}
\ee
Setting (\ref{dlF}) and (\ref{dlRm}) into (\ref{systnl}), we obtain
\be\label{dleqnl}
\begin{aligned}
& \eps\dt\Ecal(u) + \dx F(\Ecal(u)) + \eps A(\Ecal(u)) \, U_1
\\
&=
 - R(\Ecal(u)+M(0)U_1) - 
\eps B(M(0)U_1).(M(0)U_2)
\\
&\quad 
- \eps D_\eps (M(\eps)U_1)|_{\eps=0}.
  R(\Ecal(u)+M(0)U_1) + \Ocal(\eps^2).
\end{aligned}
\ee

Considering the zeroth-order terms, we get
$$
\dx F(\Ecal(u)) = -R \big( \Ecal(u)+M(0)U_1 \big),
$$
which turns out to be a nonlinear system of equations with unknown $U_1$, supplemented by 
the condition $QU_1=0$.

At this level, we see that the assumption on 
the kernel of $B(\Ecal(u))$ is no longer relevant (or sufficient)
 in order to solve (\ref{eqlin}),
since we now have to consider 
\be\label{eqU1nl}
\begin{aligned}
&R(\Ecal(u)+M(0)U_1) = -\dx F(\Ecal(u)),\\
&QU_1=0.
\end{aligned}
\ee
In view of the strong nonlinearities involved in this
equation, we tacitly assume the existence and uniqueness of the solution, denoted $\bar{U}_1$, of (\ref{eqU1nl}).
In the applications, this property will be checked directly. 

Then, we match first-order terms issuing from (\ref{dleqnl}) to get
$$
\begin{aligned}
& \dt\Ecal(u) + \dx A(\Ecal(u)).\bar{U}_1 
\\
& = -B(M(0)\bar{U}_1).(M(0)U_2)
  -D_\eps (M(\eps)\bar{U}_1)|_{\eps=0}.
  R(\Ecal(u)+M(0)\bar{U}_1).
\end{aligned}
$$
Since $Q\Ecal(u)=u$ for all $u\in\omega$, and since $QR(U)=0$ and
$QB(U)=0$ for all $U\in\Omega$, by multiplying the above relation
by $Q$, we obtain
\be\label{eqasympnl}
\dt u = - \dx \left(
QA(\Ecal(u)) \bar{U}_1
\right).
\ee

Once again, this equilibrium equation involves a
nonlinear differential operator in the right-hand side since
$\bar{U}_1$ is a nonlinear map applied to first-order space
derivatives.

Similarly to the ``linear''  case governed by (\ref{eqlim}), we want
interpret (\ref{eqasympnl}) as a system of {\it diffusion equations.} 
We thus assume the existence of a convex entropy $\Phi:\Omega\to\RR$
which satisfies all the compatibility conditions imposed in the
previous section. The matrix $D^2_U \Phi(U) D_U
F(U)$ to be symmetric for all $U\in\Omega$ and, in addition, we assume that
the compatibility conditions (\ref{cond1}) and (\ref{cond2}) hold. 
Smooth solutions to (\ref{systnl}) satisfy the additional balance law 
\be\label{entropnl}
\eps\dt\Phi(U^\eps) + \dx\Psi(U^\eps) =
-\frac{1}{\eps^m} D_U\Phi(U^\eps)R(U^\eps).
\ee

Equipped with this convex entropy, we observe that
$\bar{U}_1$ is, equivalently, a solution to the nonlinear algebraic system
$$
\begin{aligned}
&D^2_U\Phi(\Ecal(u))R(\Ecal(u)+M(0)U_1) = -D^2_U\Phi(\Ecal(u))\dx F(\Ecal(u)),\\
&QU_1=0,
\end{aligned}
$$
for any convex entropy compatible with the relaxation 
term.

Next, consider the matrix $S$ defined by (\ref{defS}) and the
vector $v$ given by (\ref{defv}). In view of 
(\ref{eqSTv}), the above system is equivalent to  
\be\label{eqSTv2}
\begin{aligned}
& D^2_U \Phi(\Ecal(u)) R(\Ecal(u)+M(0)U_1) = -S^Tv,\\
& QU_1=0.
\end{aligned}
\ee
With some abuse of notation and for the sake of clarity, we set
$$
\Ncal_u(U_1) = D^2_U \Phi(\Ecal(u)) R(\Ecal(u)+M(0)U_1) 
$$
and introduce the notation
\be\label{defNum1}
\bar{U}_1 = \Ncal^{-1}_u(-S^T v).
\ee
Hence, the equilibrium equation (\ref{eqasympnl}) reads as follows:
\be\label{diffnl}
\dt u =\dx\left(
-S\Ncal^{-1}_u(-S^T v)
\right).
\ee
Once again, we note the crucial role played by the convex
entropies. Indeed, we now exhibit the limiting system of equations satisfied by the
equilibrium entropy $\Phi(\Ecal(u))$. We skip here the details of the 
computation which are similar to the linear case. 
Lemma~\ref{lpsicons} still holds, as well as the entropy expansion
(\ref{exp1}) and (\ref{exp2}). In fact, only the entropy relaxation
source term expansion changes. We now obtain 
$$
D_U\Phi(U^\eps)R(U^\eps) = 
\eps^{m+1} \bar{U}_1^T D^2_U \Phi(\Ecal(u))
R(\Ecal(u)+M(0)\bar{U}_1).
$$

Plugging these expansions into (\ref{entropnl}) and considering first-order terms, we find 
\be\label{entropasnl}
\begin{aligned}
\dt\Phi(\Ecal(u)) =& 
-\dx\left(
D_U \Psi(\Ecal(u)) \bar{U}_1
\right)
-\bar{U}_1^T D^2_U \Phi(\Ecal(u))
R(\Ecal(u)+M(0)\bar{U}_1).
\end{aligned}
\ee
To conclude this section, we show that the asymptotic
system of equations (\ref{diffnl}) is of diffusive-type, and that an associated 
mathematical entropy is non-decreasing.

\begin{lemma}\label{lnl1}
Let $\bar{U}_1$ be given by (\ref{eqU1nl}). Assume the existence of a
non-negative map $c(u)\ge0$ such that
\be\label{asnl}
R(\Ecal(u)+M(0)\bar{U}_1) = c(u) \bar{U}_1,
\qquad u\in\omega.
\ee
Then the limiting equation (\ref{eqasympnl}) is nonlinearly dissipative
with respect to the entropy in the following sense:
\be\label{dissipnl}
-v^T S\Ncal^{-1}_u(-S^T v) \ge 0,
\qquad  u\in\omega,
\ee
where $S$ is given by (\ref{defS}) and $v$ by (\ref{defv}).

Moreover, the entropy is decreasing as follows:
\be\label{decreasnl}
\dt\Phi(\Ecal(u)) \le
-\dx\left(
D_U \Psi(\Ecal(u)).\bar{U}_1
\right).
\ee
\end{lemma}
\begin{proof}
Arguing (\ref{eqSTv2}) and the definition (\ref{defNum1}), we have
$$
-v^T S\Ncal^{-1}_u(-S^T v) = \bar{U}_1^T D^2_U \Phi(\Ecal(u))
  R(\Ecal(u)+M(0)\bar{U}_1).
$$
By recalling (\ref{asnl}), we obtain
$$
-v^T S\Ncal^{-1}_u(-S^T v) = c(u) \bar{U}_1^T D^2_U \Phi(\Ecal(u))\bar{U}_1.
$$
The inequality (\ref{dissipnl}) follows from the convexity of the entropy.
Recalling (\ref{entropasnl}) we obtain (\ref{decreasnl}), and the proof is completed.
\end{proof}


\section{Physical examples}

\subsection*{Euler equations with friction term}

Many models involving distinct physical scales enter the 
framework proposed in the present paper and, specifically,
we will now illustrate the interest of our framework with four examples.
We begin with the isentropic Euler equations supplemented with a friction term (cf.~\cite{bib_marcati,bib_vazquez} for further details). 
The asymptotics for this model has been already considered in the literarure 
and, more recently, relevant numerical techniques have been proposed \cite{bib_ounaissa,bib_ccgrs}.  
 (See also \cite{bib_marcati2,bib_bt}.) Importantly, this model satisfies all of the conditions required in 
Section~3, above. 

The Euler model with friction reads 
\be\label{Eulerfriction}
\begin{aligned}
& \eps\dt\rho +\dx\rho v=0,\\
& \eps\dt\rho v+\dx(\rho v^2+p(\rho)) = -\frac{1}{\eps} \rho v,
\end{aligned}
\ee
where $\rho>0$ denotes the density and $v\in\RR$ the 
velocity of a compressible fluid. The pressure function $p:\RR^+\to\RR^+$ is assumed to be
sufficiently regular and satisfy $p'(\rho)>0$, so that the first-order homogeneous system 
associated with \eqref{Eulerfriction} is strictly hyperbolic.

In view of (\ref{syst_1}), we should set
\be\label{notationEF}
U=\left(\begin{array}{c}
\rho \\ \rho v
\end{array}\right),
\qquad
F(U)=\left(\begin{array}{c}
\rho v \\ \rho v^2 + p(\rho)
\end{array}\right),
\quad\quad
R(U)=\left(\begin{array}{c}
0 \\ \rho v
\end{array}\right), 
\ee
which corresponds to the matrix
$$
Q=(1~0),
$$
and scalar local equilibria $u=\rho$, with 
$$
\Ecal(u) = \left(\begin{array}{c}
\rho \\ 0
\end{array}\right).
$$
As required, we also have $QF(\Ecal(u))=0$, and it is easily checked that all our assumptions of the previous sections hold. 

Considering now the asymptotic diffusive
regime in the limit $\eps \to 0$, we note that 
the equilibrium solution must satisfy (\ref{eqlim}), i.e. 
$$
\dt\rho = -\dx\left(
QA(\Ecal(u)) \, U_1
\right),
$$
where
$$
A(\Ecal(u)) = \left(\begin{array}{cc}
0 & 1 \\ p'(\rho) & 0
\end{array}\right),
$$
and $U_1$ is the unique solution to (\ref{eqU1}). Since 
$$
B(\Ecal(u)) = \left(\begin{array}{cc}
0 & 0 \\ 0 & 1
\end{array}\right), 
\quad\quad
\dx F(\Ecal(u)) = \left(\begin{array}{c}
0 \\ \dx p(\rho)
\end{array}\right),
$$
the diffusive regime associated with this Euler model
with friction is described by the equation 
\be\label{diffEuler}
\dt\rho = \partial^2_x \big( p(\rho) \big).
\ee

Based on Theorem \ref{maintheo}, we observe that the diffusive nature
of (\ref{diffEuler}) follows from the existence of a
convex entropy which is compatible with the relaxation source term in the sense
 (\ref{cond1})-(\ref{cond2}). Indeed, by introducing the internal
energy $e(\rho)>0$ defined by 
$$
e'(\rho)=\frac{p(\rho)}{\rho^2},
$$
we see that smooth solutions to (\ref{Eulerfriction}) satisfy
$$
\eps\dt\left(\rho\frac{v^2}{2}+\rho e(\rho)\right)
+
\dx\left(\rho\frac{v^2}{2}+\rho e(\rho)+p(\rho)\right)v
=-\frac{1}{\eps}\rho v^2.
$$
The function $\Phi(U)=\rho\frac{v^2}{2}+\rho e(\rho)$ is a convex
entropy satisfying the compatibility conditions with the relaxation terms.


\subsection*{The $M1$ model for radiative transfer}

The second example of interest relies on a more complex physical set-up, relevant in  
radiative transfer and referred to as the $M1$-model \cite{bib_DF,bib_minerbo}. 
(See also \cite{bib_bcd,bib_buetcord,bib_BD2,bib_goudon}.) 
This model reads 
\be\label{eqRT}
\begin{aligned}
& \eps\dt e +\dx f =\frac{1}{\eps}(\tau^4-e),\\
& \eps\dt f+\dx \Big(  \chi\left(f/e\right)e \Big) =
-\frac{1}{\eps}f,\\
& \eps\dt\tau=\frac{1}{\eps}(e-\tau^4),
\end{aligned}
\ee
where $e>0$ is the radiative energy and $f$ the radiative flux,
restricted by the ``flux limitation'' condition 
$$
\left| \frac{f}{e} \right| \le 1,
$$
and $\tau>0$ denotes the temperature. 
The function
$\chi:[-1,1]\to\RR^+$ stands for the Eddington factor defined by
$$
\chi(\xi) =\frac{ 3+4\xi^2}{5+2\sqrt{4-3\xi^2}}.
$$

Once again, we rely on the notation introduced in the previous
sections (cf.~system (\ref{syst_1})) and we write
$$
U=\left(\begin{array}{c}
e \\ f \\ \tau
\end{array}\right),
\qquad
F(U)=\left(\begin{array}{c}
f \\ \chi(\frac{f}{e})e \\ 0
\end{array}\right), 
\quad\quad
R(U)=\left(\begin{array}{c}
e-\tau^4 \\ f \\ \tau^4-e
\end{array}\right).
$$
The local equilibria are described by the map 
$$
\Ecal(u)=\left(\begin{array}{c}
\tau^4 \\ 0 \\ \tau
\end{array}\right),
$$
where $u=\tau+\tau^4$ is now a scalar. 
We set $Q=(1~0~1)$ and, in agreement with (\ref{H4_1}), we have $QF(\Ecal(u))=0$. 

The asymptotic regime is governed by (\ref{eqlim}) and to exhibit its
explicit formulation we need the expression of $A(\Ecal(u))$ and $U_1$. A straightforward
calculation gives
$$
A(\Ecal(u)) = 
\left(\begin{array}{ccc}
0 & 1 & 0 \\ 
\chi(0) & \chi'(0) & 0 \\
0 & 0 & 0
\end{array}\right)
=
\left(\begin{array}{ccc}
0 & 1 & 0 \\ 
\frac{1}{3} & 0 & 0 \\
0 & 0 & 0
\end{array}\right),
$$
while $U_1$ is the solution to (\ref{eqU1}) which here reads
$$
\begin{aligned}
\left(\begin{array}{ccc}
1 & 0 & -4\tau^3 \\ 
0 & 1 & 0 \\
-1 & 0 & 4\tau^3
\end{array}\right)U_1 & =
\left(\begin{array}{c}
0 \\ 
\dx(\frac{\tau^4}{3}) \\
0 
\end{array}\right),\\
 (1~0~1) U_1 & =0.
\end{aligned}
$$
We thus obtain
$$
U_1 = \left(\begin{array}{c}
0 \\ 
\frac{4}{3} \tau^3\dx\tau \\
0
\end{array}\right),
$$
and, according with (\ref{eqlim}), the asymptotic regime is governed
by the following diffusion equation:
$$
\dt(\tau+\tau^4) = \dx\left(
\frac{4}{3}\tau^3\dx \tau \right).
$$


\subsection*{A coupled Euler/$M1$ model}

We propose here an example of system that degenerates into a system of diffusion equations of dimension $n>1$. To do so, we couple the Euler model (\ref{Eulerfriction}) with the $M1$ model (\ref{eqRT}) as follows:
\be\label{EulerM1}
\begin{aligned}
& \eps\dt\rho +\dx\rho v=0,\\
& \eps\dt\rho v+\dx(\rho v^2+p(\rho)) = -\frac{\kappa}{\eps} \rho v+\frac{\sigma}{\eps} f,\\
& \eps\dt e +\dx f =0,\\
& \eps\dt f+\dx \chi\left(\frac{f}{e}\right)e = -\frac{\sigma}{\eps}f,
\end{aligned}
\end{equation}
in the notation previously introduced. Here, $\kappa$ and $\sigma$ denote positive constants. Even though this is a toy model, the pressure has to be sufficiently small in order to represent an application of physical 
interest. We will therefore consider the following pressure law:
\begin{equation*}
p(\rho)=C_p \rho^\eta, \qquad 
C_p\ll 1, 
\qquad 
\eta>1.
\end{equation*}
In the formalism (\ref{syst_1}) we need to set 
\begin{equation*}
U=\begin{pmatrix}\rho\\ \rho v\\ e\\f \end{pmatrix},
\qquad
F(U)=\begin{pmatrix}\rho v\\ \rho v^2+p(\rho)\\ f\\ \chi(\frac{f}{e})e \end{pmatrix},
\qquad
R(U)=\begin{pmatrix}0\\ \kappa \rho v-\sigma f\\ 0\\ \sigma f \end{pmatrix}.
\end{equation*}
The local equilibrium is given by 
\begin{equation*}
 \Ecal(u)=\begin{pmatrix}\rho\\ 0\\ e\\0 \end{pmatrix},
 \qquad
u=QU=\begin{pmatrix}\rho\\ e\end{pmatrix},
\qquad
Q=\begin{pmatrix}1&0&0&0\\0&0&1&0\end{pmatrix}.
\end{equation*}
Once again, one has $QF(\Ecal(u))=0$. 
To derive the asymptotic regime, we exhibit $A(\Ecal(u))$ and $U_1$:
\begin{equation*}
A(\Ecal(u))=\begin{pmatrix}0&1&0&0\\p'(\rho)&0&0&0\\0&0&0&1\\0&0&\frac{1}{3}&0\end{pmatrix},
\qquad
U_1=\begin{pmatrix}0\\ \frac{1}{\kappa}\bigl[-\partial_x p(\rho)-\frac{1}{3}\partial_x e\bigr]\\0\\-\frac{1}{3\sigma}\partial_x e\end{pmatrix}.
\end{equation*}
The asymptotic diffusive regime of the system (\ref{EulerM1}) is therefore given by 
\begin{equation}\label{EM1asy}
\begin{aligned}
&\partial_t \rho-\frac{1}{\kappa}\partial_x^2 p(\rho)-\frac{1}{3\kappa}\partial_x^2 e=0,\\
&\partial_t e-\frac{1}{3\sigma}\partial_x^2 e=0.
\end{aligned}
\end{equation}


\subsection*{Shallow water with strong friction effect}

The last suggested example is devoted to the well-known shallow-water
model supplemented by a strong friction term in a late-time regime where the 
friction is assumed to dominate over the convection. 

This model is given as
follows:
\be\label{SWfriction}
\begin{aligned}
& \eps \dt h +\dx hv=0,\\
& \eps \dt hv+\dx \big( h \, v^2 + p(h) \big) = -\frac{\kappa^2(h)}{\eps^2} \, g \, hv|hv|,
\end{aligned}
\ee
where $h>0$ is the water height, $v\in\RR$ the velocity and
$p(h)=g\frac{h^2}{2}$ the pressure law. Here, $g>0$ is the usual
gravity constant,
while the friction coefficient $\kappa:\RR^+ \to\RR^+$ is a given and positive function. In \cite{bib_marche}, several
examples of friction $\kappa$ are proposed. A standard choice is
$\kappa(h)=\frac{\kappa_0}{h}$ where $\kappa_0>0$ is a given
parameter.

We note that this model enters the framework of the nonlinear
extension governed by (\ref{systnl}) with $m=2$. According to the
notation involved in (\ref{systnl}), we have set
$$
U=\left(\begin{array}{c}
h \\ hv
\end{array}\right),
\qquad
F(U)=\left(\begin{array}{c}
hv \\ hv^2+p(h)
\end{array}\right), 
\quad\quad
R(U)=\left(\begin{array}{c}
0 \\ \kappa^2(h)ghv|hv|
\end{array}\right).
$$
Concerning the equilibrium, we have
$$
\Ecal(u)=\left(\begin{array}{c}
h \\ 0
\end{array}\right),
$$
where $u=h$ is a scalar and $Q=(1~0)$. The
assumption (\ref{Hadd}) easily holds since 
$$
R(\Ecal(u)+\eps U) = \eps^2 R\big( \Ecal(U) + M(\eps)U \big),
$$
where 
$$
M(\eps) = \left(\begin{array}{cc}
\eps & 0 \\ 0 & 1
\end{array}\right).
$$

We turn our attention now to the asymptotic regime which is governed by the nonlinear
diffusion equation (\ref{eqasympnl}). To get its explicit from, we
have to exhibit $U_1=(\alpha~\beta)^T$ the solution to (\ref{eqU1nl})
which reads as follows:
$$
\begin{aligned}
& \alpha = 0,\\
& \kappa^2(h)g\beta|\beta| = -\dx p(h).
\end{aligned}
$$
We easily find
$$
\beta = -  \frac{\sqrt{h}\dx h}{\kappa(h)\sqrt{|\dx h|}}, 
$$
and the effective nonlinear diffusion equation is thus 
\be\label{SWplap}
\dt h = \dx\left( {\sqrt{h} \over \kappa(h)} \, 
\frac{\dx h}{\sqrt{|\dx h|}}
\right). 
\ee
This is a nonlinear Laplacian equation (for instance, see \cite{bib_ju} and references therein).

Lemma \ref{lnl1} applies here with the following entropy. By 
introducing the internal energy $e(h)=gh/2$, we see that  smooth solutions to (\ref{SWfriction}) satisfy 
$$
\eps\dt\left( h\frac{v^2}{2}+g\frac{h^2}{2} \right)
+
\dx\left(
h\frac{v^2}{2}+gh^2
\right)v=-\frac{\kappa^2(h)}{\eps^2}ghv^2|hv|.
$$
The entropy $\Phi(U)=h\frac{v^2}{2}+g\frac{h^2}{2}$
satisfies all the required properties, and the condition (\ref{asnl}) holds since 
$$
R(\Ecal(u)+M(0)\bar{U}_1) = \left(\begin{array}{c}
0 \\ \dx p(h)
\end{array}\right),
$$
where $\bar{U}_1=(0~\beta)^T$. As a consequence, we obtain
$R(\Ecal(u)+M(0)\bar{U}_1)=c(u)\bar{U}_1$ with
$$
c(u) = g\kappa(h)\sqrt{h|\dx h|}\ge 0.
$$
Hence, the limit equation (\ref{SWplap}) is of diffusive-type in the sense of Lemma \ref{lnl1}.


\section{Asymptotic-preserving schemes} 

\subsection*{Objective}

In this section, we consider the numerical approximation of  solutions to 
(\ref{syst_1}). Our goal is to derive a class of numerical schemes that 
restore
the relevant asymptotic regime, given by (\ref{eqlim}), in the limit $\eps \to 0$.
One of the main difficulties when deriving asymptotic-preserving
schemes lies in the independent role played by each $\eps$ and the
mesh size. More precisely, the limit discrete diffusion equation (as
$\eps$ tends to zero) must be obtained independently of the space mesh-size.

Such a numerical problem was investigated during the last decade on several specific examples. For instance, in \cite{bib_ounaissa}, the
Euler equations with friction term are considered. Relating works to
radiative transfer and the $M1$-model are given in \cite{bib_buetcord,bib_goudon}. The
reader is also referred to \cite{bib_bt} where distinct physical applications
are proposed.

In the present work, we propose a generalization of a numerical scheme
derived to approximate the solutions to the Euler equations with
friction and the $M1$-model \cite{bib_bt}. To sketch this suggested
numerical procedure, we first consider a standard finite volume
method to approximate weak solutions to the homogeneous system
associated with (\ref{syst_1}) in which we have omitted $\eps$:
\be\label{syshomo}
\dt U+\dx F(U) =0.
\ee
Next, we derive a suitable
correction to obtain a finite volume discretization of the source
term. Hence, the corrected finite volume method gives approximate
solutions of the following system:
\be\label{sysRref}
\dt U+\dx F(U) = -\gamma R(U),
\ee
where $\gamma>0$ is a fixed parameter. Finally, the asymptotic
behavior of the scheme is analyzed. We
consider a late-time compatible discretization and we fix
$\gamma=\frac{1}{\eps}$. The asymptotic scheme is thus obtained in
the limit of $\eps$ to
zero. Modulo a suitable correction, it is thus
proved to be asymptotic preserving.


\subsection*{A discretization of (\ref{sysRref})}

As a first step, we suggest to consider the well-known Godunov-type scheme
introduced by Harten, Lax and van Leer \cite{bib_hlv} with a single constant
intermediate state, to approximate the weak solutions to  (\ref{syshomo}).

We consider a uniform mesh made of cells $[x_{i-\demi},x_{i+\demi})$
where $x_{i+\demi}=x_i+\frac{\Delta x}{2}$ for all $i\in\ZZ$ with a
constant cell size $\Delta x$. The time discretization is defined by
$t^{n+1}=t^n+\Delta t$ where the time increment will be restricted
later on by a CFL like condition.

We define the discrete initial data as follows:
$$
U^0(x) = \frac{1}{\Delta x}
\int_{x_{i-\demi}}^{x_{i+\demi}} U(x,0)dx,
\qquad x\in[x_{i-\demi},x_{i+\demi}).
$$
We seek for a piecewise constant approximation of the exact
solution of (\ref{syshomo}) at the time $t^n$, 
$$
U^n(x)=U_i^n, 
\qquad x\in[x_{i-\demi},x_{i+\demi}),
$$
with $U_i^n\in\Omega$ for all $i\in\ZZ$.

By considering a suitable sequence of approximate Riemann solvers, we
can evolve this approximation and get a piecewise constant
function $\tilde{U}^{n+1}(x)$ which is an approximation of the solution to 
(\ref{syshomo}) at the time $t^n+\Delta t$. Following Harten, Lax
and van Leer \cite{bib_hlv}, at each cell interface we use the following
approximate Riemann solver:
\be\label{ApRiem}
\tilde{U}_\Rcal(\frac{x}{t};U_L,U_R) =\left\{
\begin{aligned}
& U_L, \quad \frac{x}{t}<-b,\\
& \tilde{U}^\star, \quad -b<\frac{x}{t}<b,\\
& U_R, \quad \frac{x}{t}>b,
\end{aligned}\right.
\ee
where $b>0$ is a fixed and sufficiently large constant, and
\be\label{defUtild} 
\tilde{U}^\star = \frac{1}{2}(U_L+U_R)-\frac{1}{2b}
(F(U_R)-F(U_L)).
\ee
As a consequence, as soon as the following CFL restriction holds:
\be\label{cfl} 
b\frac{\Delta t}{\Delta x}\le\frac{1}{2},
\ee
we are considering a juxtaposition of non-interacting approximate
Riemann solver (cf.~Figure \ref{figRS}) denoted $\tilde{U}^n_{\Delta
  x}(x,t^n+t)$ for $t\in[0,\Delta t)$.

\begin{figure}[htb]
\begin{center}
  \centering \includegraphics[width=0.9\textwidth]{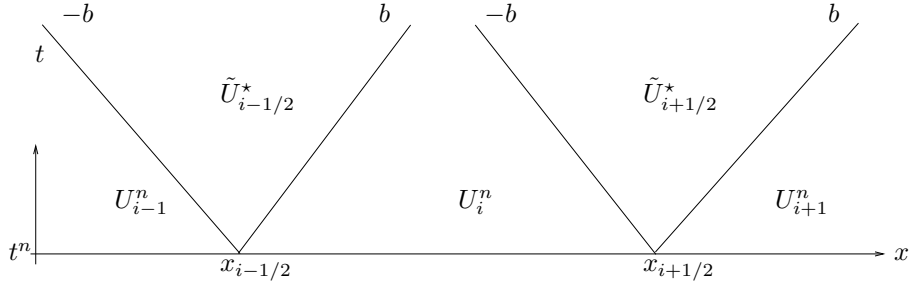}
\caption{HLL scheme: Juxtaposition of approximated Riemann problems.}
\label{figRS}
\end{center}
\begin{picture}(0,0)(0,0)
 \put(-130,60){$U_{i-1}^n$}
 \put(  0,60){$U_{i}^n$}
 \put(120,60){$U_{i+1}^n$}

 \put(-90,100){$\tilde{U}^\star_{i-\demi}$}
 \put( 70,100){$\tilde{U}^\star_{i+\demi}$}

 \put(-150 ,130){$-b$}
 \put(-30,130){$b$}
 \put(10,130){$-b$}
 \put(140,130){$b$}

 \put(-90 ,36){$x_{i-\demi}$}
 \put(70,36){$x_{i+\demi}$}

 \put(165,40){$x$}
 \put(-170,40){$t^n$}
 \put(-160,115){$t$}
\end{picture}
\end{figure}

The updated approximated solution at time $t^{n+1}$ is thus defined as
follows:
\be\label{Utildinte}
\tilde{U}_i^{n+1} = \frac{1}{\Delta x}
\int_{x_{i-\demi}}^{x_{i+\demi}} \tilde{U}^n_{\Delta x}(x,t^n+\Delta
t) dx.
\ee
Since we have
\be\label{defUtildnum} 
\tilde{U}^\star_{i+\demi} = \frac{1}{2}(U_i^n+U_{i+1}^n)-\frac{1}{2b}
(F(U_{i+1}^n)-F(U_i^n)),
\ee
an easy computation gives the following standard conservation form:
\be\label{schemehomo}
\tilde{U}_i^{n+1} = U_i^n - \frac{\Delta t}{\Delta x}
(F^{HLL}_{i+\demi} - F^{HLL}_{i-\demi}),
\ee
where we have set
\be\label{defFHLL} 
F^{HLL}_{i+\demi} = \frac{1}{2}(F(U_i^n)+F(U_{i+1}^n))
-\frac{b}{2}(U_{i+1}^n-U_i^n).
\ee

For simplicity in the presentation, we have adopted here a constant numerical cone
of dependence (cf.~Figure \ref{figRS}) characterized by a single speed 
parameter $b>0$. As prescribed in \cite{bib_hlv} (cf.~also \cite{bib_lev,bib_toro}), each
cone of dependence can be variable and defined by a pair
$(b^-_{i+\demi},b^+_{i+\demi})$ with
$b^-_{i+\demi}<b^+_{i+\demi}$. However, for the sake of simplicity and
without genuine loss of generality, we present our strategy for the simpler case 
$b^+_{i+\demi}=-b^-_{i+\demi}=b>0$.  

At this level, a first remark must be done concerning the invariant domain property for 
 the scheme (\ref{schemehomo}). It suffices to ensure that 
$\tilde{U}^\star_{i+\demi}$ belongs to $\Omega$ for all $i\in\ZZ$ to
deduce that $U_i^{n+1}$ in $\Omega$ for all $i\in\ZZ$. Indeed, from
(\ref{Utildinte}), we have
$$
\tilde{U}_i^{n+1} = b\frac{\Delta t}{\Delta x}
\tilde{U}^\star_{i-\demi} + \left(1-2b\frac{\Delta t}{\Delta x}\right)
U_i^n + b\frac{\Delta t}{\Delta x}\tilde{U}^\star_{i+\demi}
$$
and, based on the CFL restriction (\ref{cfl}), the above relation is
 a convex combination of states in $\Omega$. Since $\Omega$ is a 
convex set, we deduce that $\tilde{U}_i^{n+1}$ belongs to
$\Omega$.

Our main necessary condition for the above argument is that $\tilde{U}^\star_{i+\demi}\in\Omega$
for all $i$ in $\ZZ$. But, once again, $\tilde{U}^\star_{i+\demi}$ can be seen
as a convex combination, as follows:
$$
\tilde{U}^\star_{i+\demi}= 
\frac{1}{2}\left(U_i^n+\frac{1}{b}F(U_i^n)\right)
+ 
\frac{1}{2}\left(U_{i+1}^n-\frac{1}{b}F(U_{i+1}^n)\right).
$$
Since $\Omega$ is an open convex set, we can choose $b$ to be large enough so that to
enforce the condition $\tilde{U}^\star_{i+\demi}\in\Omega$.

Next, we modify this approximate Riemann solver and introduce a
discretization of the source-term in order to approximate the
solutions of (\ref{sysRref}). Similar modifications were made for specific problems
in \cite{bib_bt} and \cite{bib_bcd,bib_buetcord}, while we propose here a general approach based 
on matrix-valued free-parameters.
We modify the approximate Riemann solver (\ref{ApRiem}) as follows:
\be\label{ApRiemCor}
U_\Rcal(\frac{x}{t};U_L,U_R) =\left\{
\begin{aligned}
& U_L, \quad \frac{x}{t}<-b,\\
& U^{\star L}, \quad -b<\frac{x}{t}<0,\\
& U^{\star R}, \quad 0<\frac{x}{t}<b,\\
& U_R, \quad \frac{x}{t}>b,
\end{aligned}\right.
\ee
where we have set
\be\label{defUstar}
\begin{aligned}
& U^{\star L} = \underline{\alpha}\tilde{U}^\star + 
    (I-\underline{\alpha})(U_L-\bar{R}(U_L)),\\
& U^{\star R} = \underline{\alpha}\tilde{U}^\star + 
    (I-\underline{\alpha})(U_R-\bar{R}(U_R)).
\end{aligned}
\ee
Here, $\underline{\alpha}$ denotes a $N\times N$ matrix defined as
follows:
\be\label{defsbar}
\underline{\alpha} = \left( I +\frac{\gamma\Delta
  x}{2b}(I+\underline{\sigma})
\right)^{-1},
\ee
and
\be\label{defRbar}
\bar{R}(U) = (I+\underline{\sigma})^{-1} R(U).
\ee
The $N\times N$ matrices $I$ and $\underline{\sigma}$ respectively
denote the identity matrix and a parameter matrix to be defined.

A choice for the matrix $\underline{\sigma}$ will be made later, 
which will turn out to govern 
the asymptotic regime.  At this point, the matrix $\underline{\sigma}$ is assumed to 
be such 
that the inverse matrices in (\ref{defsbar}) and
(\ref{defRbar}) are well-defined.

We adopt the modified approximate Riemann solver (\ref{ApRiemCor}) to
derive a modified Godunov type scheme in the spirit of \cite{bib_bt}. At
each cell interface $x_{i+\demi}$, we set the approximate Riemann
solver $U_{\Rcal}(\frac{x-x_{i+\demi}}{t-t^n}; U_i^n,U_{i+1}^n)$ to
define a juxtaposition of modified approximate Riemann solver,
denoted by $U^n_{\Delta x}(x,t^n+t)$ for
$t\in[0,\Delta t)$. (See Figure \ref{figRSC}.) 

\begin{figure}[htb]
\begin{center}
  \centering \includegraphics[width=0.9\textwidth]{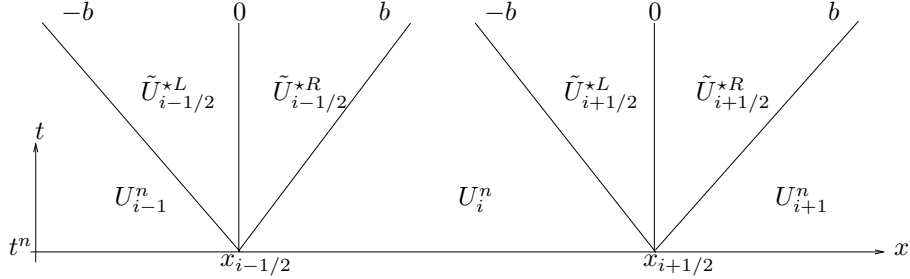}
\caption{Juxtaposition of modified approximated Riemann problems.}
\label{figRSC}
\end{center}
\begin{picture}(0,0)(0,0)
 \put(-130,60){$U_{i-1}^n$}
 \put(  0,60){$U_{i}^n$}
 \put(120,60){$U_{i+1}^n$}

 \put(-120,100){$\tilde{U}^{\star L}_{i-\demi}$}
 \put(-70,100){$\tilde{U}^{\star R}_{i-\demi}$}
 \put( 40,100){$\tilde{U}^{\star L}_{i+\demi}$}
 \put( 90,100){$\tilde{U}^{\star R}_{i+\demi}$}

 \put(-150 ,130){$-b$}
 \put(-30,130){$b$}
 \put(10,130){$-b$}
 \put(140,130){$b$}

 \put(-85,130){$0$}
 \put( 72,130){$0$}

 \put(-90,37){$x_{i-\demi}$}
 \put(70, 37){$x_{i+\demi}$}

 \put(165,41){$x$}
 \put(-170,41){$t^n$}
 \put(-160,85){$t$}
\end{picture}
\end{figure}

Thanks to the CFL condition (\ref{cfl}), such a juxtaposition is
non-interacting.

At time $t^{n+1}$, the updated approximated solution is given as
follows for all $i\in\ZZ$:
\be\label{Un1inte}
U_i^{n+1} = 
\int_{x_{i-\demi}}^{x_{i+\demi}} U^n_{\Delta x}(x,t^n+\Delta t)dx.
\ee
We compute this integral form and obtain 
\be\label{schaux}
\begin{aligned}
& \frac{1}{\Delta t}(U_i^{n+1}-U_i^n) +\frac{1}{\Delta x}
(\underline{\alpha}_{i+\demi} F^{HLL}_{i+\demi}
-\underline{\alpha}_{i-\demi} F^{HLL}_{i-\demi})
\\
& =
\frac{1}{\Delta x}
  (\underline{\alpha}_{i+\demi}-\underline{\alpha}_{i-\demi}) F(U_i^n)
-\frac{b}{\Delta x} (I-\underline{\alpha}_{i-\demi})
  \bar{R}_{i-\demi}(U_i^n)
\\
& \quad - 
\frac{b}{\Delta x} (I-\underline{\alpha}_{i+\demi})
  \bar{R}_{i+\demi}(U_i^n),
\end{aligned}
\ee
where the numerical flux  $F^{HLL}_{i+\demi}$ is given by
(\ref{defFHLL}). The discretized source-term can be rewritten in a more
relevant form: 
$$
\frac{b}{\Delta x}(I-\underline{\alpha}_{i+\demi})
  \bar{R}_{i+\demi}(U_i^n) = 
\frac{b}{\Delta x}\underline{\alpha}_{i+\demi}
 (\underline{\alpha}_{i+\demi}^{-1}-I)
  \bar{R}_{i+\demi}(U_i^n).
$$
In view of the definition of $\underline{\alpha}_{i+\demi}$ (see
(\ref{defsbar})), we deduce that 
$$
\frac{b}{\Delta x}(I-\underline{\alpha}_{i+\demi})
  \bar{R}_{i+\demi}(U_i^n) = 
\frac{\gamma}{2}\underline{\alpha}_{i+\demi}R(U_i^n) 
$$
and, similarly, 
$$
\frac{b}{\Delta x}(I-\underline{\alpha}_{i-\demi})
  \bar{R}_{i-\demi}(U_i^n) = 
\frac{\gamma}{2}\underline{\alpha}_{i-\demi}R(U_i^n).
$$
As a consequence, the scheme (\ref{schaux}) reads 
\be\label{schemeR}
\begin{aligned}
& \frac{1}{\Delta t}(U_i^{n+1}-U_i^n) + \frac{1}{\Delta x}
(\underline{\alpha}_{i+\demi} F^{HLL}_{i+\demi}
-\underline{\alpha}_{i-\demi} F^{HLL}_{i-\demi})
\\
& =
\frac{1}{\Delta x}
  (\underline{\alpha}_{i+\demi}-\underline{\alpha}_{i-\demi}) F(U_i^n)
-\frac{\gamma}{2}
(\underline{\alpha}_{i+\demi}+\underline{\alpha}_{i-\demi})R(U_i^n).
\end{aligned}
\ee
This scheme satisfies the following statement.
\begin{theorem}\label{theoconsis}
Assume that the matrix $\underline{\sigma}$ defines a spatially continuous
map. The numerical scheme (\ref{schemeR}) is consistant with the
equation (\ref{sysRref}).

At time $t^n$,  assume $U_i^n\in\Omega$ for all $i\in\ZZ$ and, in 
addition, that the state vectors $U^{\star L}_{i+\demi}$ and
$U^{\star R}_{i+\demi}$, defined by
$$
\begin{aligned}
&  U^{\star L}_{i+\demi} = \underline{\alpha}_{i+\demi}\tilde{U}^\star_{i+\demi} + 
    (I-\underline{\alpha}_{i+\demi})(U_i^n-\bar{R}(U_i^n)),\\
& U^{\star R}_{i+\demi} = \underline{\alpha}_{i+\demi}\tilde{U}^\star_{i+\demi} + 
    (I-\underline{\alpha}_{i+\demi})(U_{i+1}^n-\bar{R}(U_{i+1}^n)),
\end{aligned}
$$
belong to $\Omega$. Then, the updated state vector $U_i^{n+1}$,
defined by (\ref{schemeR}), belongs the set $\Omega$ for all $i$ in $\ZZ$.
\end{theorem}

\begin{proof}
The consistency property follows from the definition of
$\underline{\alpha}_{i+\demi}$, given by (\ref{defsbar}). Indeed, we
easily obtain
$$
\underline{\alpha}_{i+\demi} = I +\Ocal(\Delta x),
$$
to deduce the expected consistency of the flux  and the
relaxation source term. Only the term $\frac{1}{\Delta x}
  (\underline{\alpha}_{i+\demi}-\underline{\alpha}_{i-\demi}) F(U_i^n)$
is left over. Recalling (\ref{defsbar}), we have
$$
\frac{1}{\Delta x}
  (\underline{\alpha}_{i+\demi}-\underline{\alpha}_{i-\demi}) F(U_i^n) = 
-\frac{\gamma}{2b}\underline{\alpha}_{i+\demi}
(\underline{\sigma}_{i+\demi}-\underline{\sigma}_{i-\demi})
\underline{\alpha}_{i-\demi}F(U_i^n),
$$
to obtain
$$
\frac{1}{\Delta x}
  (\underline{\alpha}_{i+\demi}-\underline{\alpha}_{i-\demi}) F(U_i^n) = 
\Ocal(\Delta x),
$$
as soon as
$\underline{\sigma}_{i+\demi}-\underline{\sigma}_{i-\demi}=\Ocal(\Delta
x)$. The expected equation consistency is therefore obtained.

Concerning the robustness of the method, from (\ref{Un1inte}), we have
$$
U_i^{n+1} =  
b\frac{\Delta t}{\Delta x}U^{\star L}_{i-\demi} + 
\left(1-2b\frac{\Delta t}{\Delta x}\right) U_i^n + 
b\frac{\Delta t}{\Delta x}U^{\star R}_{i+\demi},
$$
which is nothing but a convex combination in $\Omega$. Then
$U_i^{n+1}$ is in $\Omega$ and the proof is completed.
\end{proof}

To conclude this derivation, observe that the term 
$\frac{1}{\Delta x}
  (\underline{\alpha}_{i+\demi}-\underline{\alpha}_{i-\demi}) F(U_i^n)$
may seem to be a discrepancy in the method. In fact, this term is
standard to derive asymptotic preserving schemes and it can be found in
several works. (See, for instance, \cite{bib_bcd,bib_bt,bib_buetcord}.) 
Our approach allows us recover a scheme proposed earlier in 
\cite{bib_gosse}. 


\subsection*{The linear asymptotic regime}

The scheme (\ref{schemeR}) is now considered to suggest a
discretization of our initial model (\ref{syst_1}). The expected scheme
is thus easily obtained when substituting $\Delta t$ by $\frac{\Delta
  t}{\eps}$ and fixing $\gamma=\frac{1}{\eps}$. The resulting
scheme reads as follows:
\be\label{schR1}
\begin{aligned}
& \frac{\eps}{\Delta t}(U_i^{n+1}-U_i^n) +\frac{1}{\Delta x}
(\underline{\alpha}_{i+\demi} F^{HLL}_{i+\demi}
-\underline{\alpha}_{i-\demi} F^{HLL}_{i-\demi})
\\
&
= \frac{1}{\Delta x}
  (\underline{\alpha}_{i+\demi}-\underline{\alpha}_{i-\demi}) F(U_i^n)
-\frac{1}{2\eps}
(\underline{\alpha}_{i+\demi}+\underline{\alpha}_{i-\demi})R(U_i^n),
\end{aligned}
\ee
where
\be\label{defalphanum}
\underline{\alpha}_{i+\demi}=\left(
I+\frac{\Delta x}{2\eps b}(I+\underline{\sigma}_{i+\demi})
\right)^{-1}.
\ee
For the sake of simplicity in the forthcoming asymptotic derivation,
we propose to introduce the $N\times N$ matrix
$\underline{\alpha}_{i+\demi}^\eps$ defined by
\be\label{defaleps}
\underline{\alpha}_{i+\demi}^\eps= \left(
\eps I+\frac{\Delta x}{2 b}(I+\underline{\sigma}_{i+\demi})
\right)^{-1},
\ee
so that we have
$\underline{\alpha}_{i+\demi}=\eps\underline{\alpha}_{i+\demi}^\eps$.
Recalling the definition, the scheme (\ref{schemeR}) takes the form:
\be\label{scheps}
\begin{aligned}
& \frac{\eps}{\Delta t}(U_i^{n+1}-U_i^n) +\frac{\eps}{\Delta x}
(\underline{\alpha}_{i+\demi}^\eps F^{HLL}_{i+\demi}
-\underline{\alpha}_{i-\demi}^\eps F^{HLL}_{i-\demi})
\\&
= \frac{\eps}{\Delta x}
  (\underline{\alpha}_{i+\demi}^\eps-\underline{\alpha}_{i-\demi}^\eps) F(U_i^n)
-\frac{1}{2}
(\underline{\alpha}_{i+\demi}^\eps+\underline{\alpha}_{i-\demi}^\eps)R(U_i^n).
\end{aligned}
\ee
We observe that $U_i^n$ remains close to the equilibrium state
$\Ecal(u_i^n)$ for $\eps$ small and we thus consider the following expansion:
$$
U_i^n = \Ecal(u_i^n) + \eps(U_1)_i^n + \Ocal(\eps^2),
$$
which we now plug into (\ref{scheps}). We easily have
$$
\begin{aligned}
& \frac{1}{\eps}R(U_i^n) = B(\Ecal(u_i^n)).(U_1)_i^n +
  \Ocal(\eps),\\
& F^{HLL}_{i+\demi} |_{\Ecal(u)+\Ocal(\eps)} =
   F^{HLL}_{i+\demi} |_{\Ecal(u)} + \Ocal(\eps),
\end{aligned}
$$
where 
$$
F^{HLL}_{i+\demi}|_{\Ecal(u)} =
{1 \over 2} \left( F(\Ecal(u_i^n)) + F(\Ecal(u_{i+1}^n)) \right)
-\frac{b}{2}\left( \Ecal(u_{i+1}^n) - \Ecal(u_i^n) \right).
$$
In addition we have
$$
\underline{\alpha}_{i+\demi}^\eps=
\frac{2b}{\Delta x} (I+\underline{\sigma}_{i+\demi})^{-1} + \Ocal(\eps).
$$
By considering the first-order terms issuing from (\ref{scheps}), we
obtain
$$
\begin{aligned}
& \frac{1}{\Delta t}(\Ecal(u_i^{n+1}) - \Ecal(u_i^n))
\\
& + \frac{2b}{\Delta x^2}\left(
(I+\underline{\sigma}_{i+\demi})^{-1} F^{HLL}_{i+\demi}|_{\Ecal(u)}
-
(I+\underline{\sigma}_{i-\demi})^{-1} F^{HLL}_{i-\demi}|_{\Ecal(u)}
\right)
\\
& = 
\frac{2b}{\Delta x^2} \left(
(I+\underline{\sigma}_{i+\demi})^{-1}
-
(I+\underline{\sigma}_{i-\demi})^{-1}
\right) F(\Ecal(u_i^n))
\\&
\quad - \frac{b}{\Delta x} \left(
(I+\underline{\sigma}_{i+\demi})^{-1}
+
(I+\underline{\sigma}_{i-\demi})^{-1}
\right)B(\Ecal(u_i^n)).(U_1)_i^n.
\end{aligned}
$$
Next,  assume the existence of a $n\times n$ squared matrix, denoted by
$\Mcal_{i+\demi}$, such that
\be\label{HM}
Q(I+\underline{\sigma}_{i+\demi})^{-1}=\frac{1}{b^2}\Mcal_{i+\demi} Q.
\ee
Recalling the assumptions (\ref{H1_1}), (\ref{H3_1}), and (\ref{H4_1}), we
 write
$$
\frac{1}{\Delta t}(u_i^{n+1}-u_i^n) =
-\frac{2}{b\Delta x^2}\left(
\Mcal_{i+\demi} Q F^{HLL}_{i+\demi}|_{\Ecal(u)} - 
\Mcal_{i-\demi} Q F^{HLL}_{i-\demi}|_{\Ecal(u)} \right),
$$
where 
$$
\begin{aligned}
Q F^{HLL}_{i+\demi}|_{\Ecal(u)}&= {1 \over 2} Q 
\left( F(\Ecal(u_i^n)) + F(\Ecal(u_{i+1}^n)) \right)
-\frac{b}{2} Q \left(\Ecal(u_{i+1}^n)-\Ecal(u_i^n) \right)
\\
&= -\frac{b}{2} (u_{i+1}^n -u_i^n).
\end{aligned}
$$
As a consequence, the asymptotic discrete regime is given by
\be\label{schasymp}
\frac{1}{\Delta t}(u_i^{n+1}-u_i^n) =
\frac{1}{\Delta x^2}
\left(
\Mcal_{i+\demi}(u_{i+1}^n -u_i^n) + 
\Mcal_{i-\demi}(u_{i-1}^n -u_i^n)
\right).
\ee
We thus have established the following result.

\begin{theorem}\label{theoasympnum}
Consider any $N\times N$ matrix $\underline{\sigma}_{i+\demi}$
such that the matrices $I+\underline{\sigma}_{i+\demi}$ and
$(1+\frac{\Delta x}{2\eps b})I+\underline{\sigma}_{i+\demi}$ are
nonsingular for all $\eps>0$. Assume the existence of a $n\times n$
matrix $\Mcal_{i+\demi}$ such that (\ref{HM}) holds, and introduce
the $n\times n$ matrix $\Mcal(u)$ defined by 
$$
\Mcal(u) = Q
A(\Ecal(u)) \Lcal^{-1}(u) D^2_U \Phi(\Ecal(u)) A(\Ecal(u))
D_u \Ecal(u).
$$
In addition, assume that the matrix $\Mcal_{i+\demi}$ is a discrete form
of $\Mcal(u)$ at each cell interface $x_{i+\demi}$. Then, the
asymptotic behavior of the scheme (\ref{scheps}) coincides with a
discrete form for the limit diffusion equation (\ref{eqlim2}).
\end{theorem}

\begin{proof}
We directly deduce from (\ref{eqlim2}) that the diffusive limit
equation reads 
$$
\dt u = \dx \left(\Mcal(u) \dx u \right).
$$
Since the asymptotic regime satisfied by the scheme (\ref{scheps}) is
governed by (\ref{schasymp}), the proposed choice of the matrix $\Mcal_{i+\demi}$ 
leads us to the correct behavior of the scheme as $\eps$ goes to zero.
\end{proof}


\subsection*{The nonlinear asymptotic regime}

We propose to extend the above numerical scheme to consider the
nonlinear asymptotic regime governed by the system (\ref{systnl}). To
address such an issue, once again we consider the scheme
(\ref{schemeR}) where $\Delta t$ is substituted by $\frac{\Delta
t}{\eps}$ and we set $\gamma=\frac{1}{\eps}$. To be
consistent, we substitute $R(U_i^n)$ by
$\frac{1}{\eps^{m-1}}R(U_i^n)$.

Adopting such a strategy, the same arguments used to obtain
(\ref{scheps}) now give
\be\label{schepsnl}
\begin{aligned}
& \frac{1}{\Delta t}(U_i^{n+1}-U_i^n) + \frac{1}{\Delta x}
(\underline{\alpha}_{i+\demi}^\eps F^{HLL}_{i+\demi}
-\underline{\alpha}_{i-\demi}^\eps F^{HLL}_{i-\demi})
\\
& = 
\frac{1}{\Delta x}
  (\underline{\alpha}_{i+\demi}^\eps-\underline{\alpha}_{i-\demi}^\eps) F(U_i^n)
-\frac{1}{2\eps^m}
(\underline{\alpha}_{i+\demi}^\eps+\underline{\alpha}_{i-\demi}^\eps)R(U_i^n).
\end{aligned}
\ee
where $\underline{\alpha}_{i+\demi}^\eps$ is defined by
(\ref{defaleps}). The above scheme exactly
coincides with (\ref{scheps}) as soon as we fix $m=1$.

Once again, as soon as $\eps$ tends to zero, the state vector
$U_i^n$ remains in a neighborhood of the equilibrium state
$\Ecal(u_i^n)$. As a consequence, we adopt a formal expansion given by
$$
U_i^n = \Ecal(u_i^n) +\eps (U_1)_i^n +\Ocal(\eps^2).
$$
Arguing the same calculations as used in the linear expansion case,
the asymptotic discrete equation is given by (\ref{schasymp}). Hence,
we have to choose $\Mcal_{i+\demi}$ in order to get a discretization
of (\ref{eqasympnl}).

\begin{theorem}
Consider an $N\times N$ matrix $\underline{\sigma}_{i+\demi}$
such that the matrices $I+\underline{\sigma}_{i+\demi}$ and
$(1+\frac{\Delta x}{2\eps b})I+\underline{\sigma}_{i+\demi}$ are
non-singular for all $\eps>0$. Assume the existence of a $n\times n$
matrix $\Mcal_{i+\demi}$ such that (\ref{HM}) holds.
Consider also the $n\times n$ matrix $\Mcal(u)$ defined by 
\be\label{Munl}
\Mcal(u) = Q A(\Ecal(u)) \Rcal^{-1}(u),
\ee
where $\Rcal^{-1}:\omega\to\Omega$ defines the unique solution of
(\ref{eqU1nl}). Assume that $\Mcal_{i+\demi}$ is a discrete form of
$\Mcal(u)$ at each cell interface $x_{i+\demi}$. The asymptotic
behavior of the scheme (\ref{schepsnl}) defines a discrete form of the
nonlinear diffusion equation (\ref{eqasympnl}).
\end{theorem}

\begin{proof}
By definition of $\Rcal^{-1}(u)$, we have $\bar{U}_1=\Rcal^{-1}(u)$
the unique solution to (\ref{eqU1nl}). As a consequence, we deduce the
following rewriting of (\ref{eqasympnl}):
$$
\dt u = \dx\left(\Mcal(u)\dx u\right),
$$
where $\Mcal(u)$ is given by (\ref{Munl}). The proposed definition of
$\Mcal_{i+\demi}$ concludes the proof.
\end{proof}


\section{Numerical experiments}

\subsection*{Euler equations with friction}

To illustrate the interest of the asymptotic preserving scheme (\ref{schR1}), we
apply it now to the Euler equations with a friction term
(\ref{Eulerfriction}). Here, we suggest to fix the matrix parameter 
$\underline{\sigma}_{i+\demi}$ as follows:
$$
\underline{\sigma}_{i+\demi}= \sigma_{i+\demi} I,
$$
where $\sigma_{i+\demi}$ stands for a scalar parameter to be
defined. As a consequence, the matrix $\underline{\alpha}_{i+\demi}$,
defined by (\ref{defalphanum}), is now given by
\be\label{defalphascal}
\underline{\alpha}_{i+\demi} = \alpha_{i+\demi} I, 
\qquad\quad
\alpha_{i+\demi}= \frac{1}{1+\frac{\Delta x}{2\eps b}
(1+\sigma_{i+\demi})}.
\ee
The scheme (\ref{schR1}) thus reads 
\begin{align}
\label{schEl1}
& \begin{aligned}
& \frac{\eps}{\Delta t}(\rho_i^{n+1} - \rho_i^n)
 + \frac{1}{\Delta x} \left(
\alpha_{i+\demi} F^{\rho,HLL}_{i+\demi} -
 \alpha_{i-\demi} F^{\rho,HLL}_{i-\demi}
\right)
\\
& = -\alpha_{i+\demi}
\frac{\sigma_{i+\demi}-\sigma_{i-\demi}}{2\eps b}
\alpha_{i-\demi} (\rho v)_i^n,
\end{aligned}\\
\label{schEl2}
& \begin{aligned}
& \frac{\eps}{\Delta t}((\rho v)_i^{n+1} - (\rho v)_i^n)
 + \frac{1}{\Delta x} \left(
\alpha_{i+\demi} F^{\rho v,HLL}_{i+\demi} -
  \alpha_{i-\demi} F^{\rho v,HLL}_{i-\demi}
\right)
\\
& = -\alpha_{i+\demi}
\frac{\sigma_{i+\demi}-\sigma_{i-\demi}}{2\eps b}
\alpha_{i-\demi} \left(\rho_i^n (v_i^n)^2 + p(\rho_i^n)\right)
\\
 & 
-\frac{1}{\eps} \frac{\alpha_{i+\demi}+\alpha_{i-\demi}}{2}
(\rho v)_i^n,
\end{aligned}
\end{align}
where the numerical flux vector $( F^{\rho,HLL}_{i+\demi}, F^{\rho v,HLL}_{i+\demi})$
is defined by (\ref{defFHLL}).

First, applying Theorem \ref{theoconsis}, we observe that 
the proposed
scheme is consistent with the system (\ref{Eulerfriction}) and 
preserves the admissible states. To address such an issue, in view of 
Theorem \ref{theoconsis}, we have to establish  the
positivity property ($i\in\ZZ$) 
$$
\begin{aligned}
& \rho_{i+\demi}^{\star L} = \alpha_{i+\demi}
\tilde{\rho}_{i+\demi}^\star
  +(1-\alpha_{i+\demi})\rho_i^n,\\
& \rho_{i+\demi}^{\star R} = \alpha_{i+\demi}
\tilde{\rho}_{i+\demi}^\star
  +(1-\alpha_{i+\demi})\rho_{i+1}^n.
\end{aligned}
$$
Since $\alpha_{i+\demi}$, defined by (\ref{defalphascal}), belongs to 
$(0,1)$, we have $\rho_{i+\demi}^{\star L}>0$ and
$\rho_{i+\demi}^{\star R}>0$ as soon as
$\tilde{\rho}_{i+\demi}^\star>0$ which is satisfied for sufficiently large 
values of $b$ (cf.~relation (\ref{defUtild})).

Now, we study for the asymptotic behavior of the scheme
(\ref{schEl1})-(\ref{schEl2}). Put in other words, we have to fix the
free parameter $\sigma_{i+\demi}$ to recover the expected asymptotic
regime in the limit of $\eps$ to zero. This required
asymptotic behavior must be governed by a discrete form of (\ref{diffEuler}). First, 
observe that 
$$
\lim_{\eps\to 0}
  \alpha_{i+\demi} =0, 
\qquad
\lim_{\eps\to 0}
  \frac{\alpha_{i+\demi}} 
{\varepsilon} =0
$$
From the momentum $(\rho v)_i^{n+1}$ governing equation
(\ref{schEl1}), we easily deduce the following momentum behavior in
the limit of $\eps$ to zero:
$$
(\rho v)_i^n =0,\qquad  i\in\ZZ.
$$
The density approximation thus admits the following asymptotic regime:
$$
\frac{1}{\Delta t}(\rho_i^{n+1}-\rho_i^n)
+\frac{2b}{\delta x^2}\left(
\frac{1}{1+\sigma_{i+\demi}} F^{\rho,HLL}_{i+\demi}|_{\rho v=0}
-
\frac{1}{1+\sigma_{i-\demi}} F^{\rho,HLL}_{i-\demi}|_{\rho v=0}
\right)=0.
$$
But we have
$$
 F^{\rho,HLL}_{i-\demi}|_{\rho v=0} =
-\frac{b}{2}\left( \rho_{i+1}^n - \rho_i^n \right),
$$
and therefore 
$$
\frac{1}{\Delta t}(\rho_i^{n+1}-\rho_i^n)
=\frac{b^2}{\delta x^2}\left(
\frac{1}{1+\sigma_{i+\demi}} (\rho_{i+1}^n - \rho_i^n)
+
\frac{1}{1+\sigma_{i-\demi}} (\rho_{i-1}^n - \rho_i^n)
\right).
$$
We choose
\be
\label{corrEl}
\sigma_{i+\demi} = \left\{
\begin{aligned}
& b^2\frac{\rho_{i+1}^n - \rho_i^n}{p(\rho_{i+1}^n) - p(\rho_i^n)}
   -1,\qquad\mbox{ if }\rho_{i+1}^n\not=\rho_i^n,\\
& \frac{b^2}{p'(\rho_i^n)}-1,
  \qquad\mbox{ otherwise},
\end{aligned}\right.
\ee
and arrive at the following discretization of the diffusion equation
(\ref{diffEuler}):
$$
\frac{1}{\Delta t}(\rho_i^{n+1}-\rho_i^n)
=\frac{1}{\delta x^2}\left(
p(\rho_{i+1}^n) - 2p(\rho_i^n) + p(\rho_{i-1}^n) 
\right).
$$
To conclude, observe that Theorem \ref{theoasympnum}
applies if the matrix $\Mcal_{i+\demi}$ is defined by
$\Mcal_{i+\demi}=m_{i+\demi} I_n$ where $I_n$ denotes the $n\times n$ identity
matrix and
$$
m_{i+\demi} = \left\{
\begin{aligned}
& \frac{p(\rho_{i+1}^n) - p(\rho_i^n)}{\rho_{i+1}^n - \rho_i^n},
  \qquad\mbox{ if }\rho_{i+1}^n\not=\rho_i^n,\\
& p'(\rho_i^n),
  \qquad\mbox{ otherwise}.
\end{aligned}\right.
$$

In \cite{bib_bt,bib_ccgrs}, the scheme (\ref{schEl1})-(\ref{schEl2}) was derived by a completely different approach.
Similarly, in \cite{bib_bt,bib_bcd,bib_buetcord}, the application of our general asymptotic preserving scheme (\ref{schR1})
is found in the framework of the radiative transfer (\ref{eqRT}).

\begin{figure}[htb]
  \centering
\scalebox{0.4}{\rotatebox{270}{\includegraphics{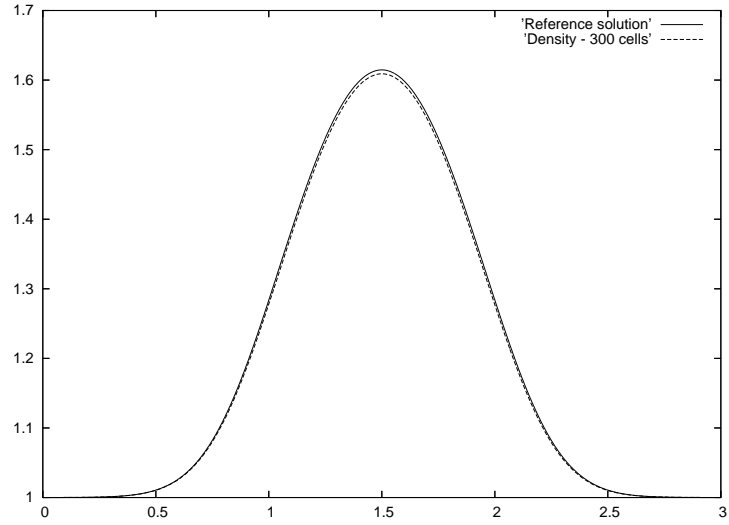}}}\\
\caption{\label{rho_comp}Reference (full line) and proposed scheme (dashed line) solution comparison.}
\end{figure}

\begin{figure}[htb]
  \centering
\scalebox{0.4}{\rotatebox{270}{\includegraphics{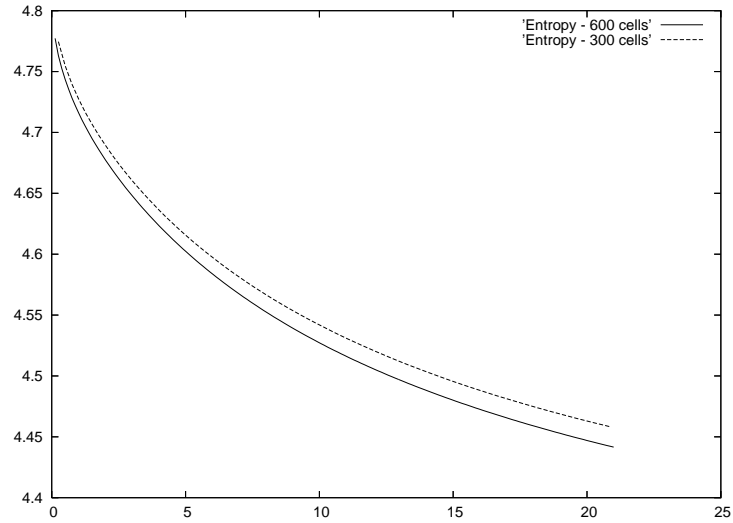}}}\\
\caption{\label{entropie}Decrease of the entropy.}
\end{figure}

As an illustration, the scheme (\ref{schEl1})-(\ref{schEl2}) supplemented by the asymptotic preserving correction 
(\ref{corrEl}) is used to approximate the solution when the initial data is given by
$$
(\rho,\rho v)^T=\begin{cases}(2,0)^T, \quad \text{$x\in[1.2,1.8]$},\\(1,0)^T, \quad \text{otherwise}.\end{cases}
$$
Furthermore, we choose the simple pressure law $p(\rho)=\rho^2$ and $\eps=10^{-3}$. In Figure~\ref{rho_comp}, the solution on a 300 cells grid ($\Delta x=10^{-2}$) is compared at time $t=2.10^4$ with a numerical approximation of (\ref{diffEuler}) computed 
with 
600 cells. We note a fairly good agreement between the two approximate results.
In agreement with
 Lemma~\ref{lentrop}, Figure~\ref{entropie} shows that entropy 
is decreasing in time. For two choices of space grids, 
we plot $\sum_i\limits \Bigl(\rho \frac{v^2}{2}+\rho e(\rho)\Bigr)_i$ versus time.


\subsection*{Coupled Euler/$M1$ equations}

We now propose to approximate the solution of the system (\ref{EulerM1}) by adopting the asymptotic-preserving scheme (\ref{schR1}) with the following matrix parameter $\underline{\sigma}_{i+\demi}$:
\begin{equation*}
\underline{\sigma}_{i+\demi}=\begin{pmatrix}\underline{\sigma}_{1,i+\demi}&0&-\underline{\sigma}_{2,i+\demi}&0\\0&0&0&0\\0&0&\underline{\sigma}_{3,i+\demi}&0\\0&0&0&0\end{pmatrix},
\end{equation*}
where $\underline{\sigma}_{j,i+\demi}$ are parameters to be defined later, 
in order
 to reach the required asymptotic regime (\ref{EM1asy}). With such a definition $\underline{\alpha}^\e_{i+\demi}$, defined by (\ref{defaleps}), becomes:
\begin{align*}
\underline{\alpha}^\e_{k}=\begin{pmatrix}\dfrac{2b_{k}\e}{\theta_{1,k}}&0&\dfrac{2b_{k}\e\gamma\Delta x\underline{\sigma}_{2,k}}{\theta_{1,k}\theta_{3,k}}&0\\0&\dfrac{2b_{k}\e}{2b_{k}\e+\gamma\Delta x}&0&0\\0&0&\dfrac{2b_{k}\e}{\theta_{3,k}}&0\\0&0&0&\dfrac{2b_{k}\e}{2b_{k}\e+\gamma\Delta x}\end{pmatrix},
\end{align*}
where $\theta_{j,k}=2b_{k}\e+\gamma\Delta x(1+\underline{\sigma}_{j,k})$ and $\gamma=\frac{\max(\kappa,\sigma)}{\e}$. The asymptotic regime associated with this scheme is 
\begin{align*}
\rho_i^{n+1}&
= \rho_i^n-\dfrac{\Delta t}{\Delta x}\dfrac{b^2}{\gamma\Delta x}\Bigg(
\dfrac{\rho_{i+1}^n-\rho_i^n}{1+\underline{\sigma}_{1,i+\demi}}-\dfrac{\rho_i^n-\rho_{i-1}^n}{1+\underline{\sigma}_{1,i-\demi}}\\
 &\quad \hskip3.cm 
 +\dfrac{\underline{\sigma}_{2,i+\demi}(e_{i+1}^n-e_i^n)}{(1+\underline{\sigma}_{1,i+\demi})(1+\underline{\sigma}_{3,i+\demi})}-\dfrac{\underline{\sigma}_{2,i-\demi}(e_{i}^n-e_{i-1}^n)}{(1+\underline{\sigma}_{1,i-\demi})(1+\underline{\sigma}_{3,i-\demi})}\Bigg),
 \\
e_i^{n+1}&=e_i^n-\dfrac{\Delta t}{\Delta x}\dfrac{b^2}{\gamma\Delta x}\Bigg(
\dfrac{e_{i+1}^n-e_i^n}{1+\underline{\sigma}_{3,i+\demi}}-\dfrac{e_{i}^n-e_{i-1}^n}{1+\underline{\sigma}_{3,i-\demi}}\Bigg).
\end{align*}
To recover a discrete form of (\ref{EM1asy}), we propose to set 
\begin{align*}
\underline{\sigma}_{1,i+\demi}&=\dfrac{\kappa}{\gamma}\dfrac{\rho_{i+1}^n-\rho_i^n}{p_{i+1}^n-p_i^n}-1,\\
\underline{\sigma}_{2,i+\demi}&=\dfrac{\sigma}{\gamma}\dfrac{\rho_{i+1}^n-\rho_i^n}{p_{i+1}^n-p_i^n},\\
\underline{\sigma}_{3,i+\demi}&=\dfrac{3\sigma}{\gamma}-1.
\end{align*}
In conclusion, setting $\Delta p_{i+\demi}^n = \dfrac{p_{i+1}^n-p_i^n}{\rho_{i+1}^n-\rho_i^n}$ we obtain
\begin{equation*}
\underline{\sigma}_{i+\demi}=\begin{pmatrix}{\kappa \over \gamma \Delta p_{i+\demi}^n}-1&0&-{\sigma \over \gamma \Delta p_{i+\demi}^n}&0
\\0&0&0&0\\0&0&\dfrac{3\sigma}{\gamma}-1&0\\0&0&0&0\end{pmatrix},
\end{equation*}

This scheme is now applied to the following numerical experiment. We choose an initial data given by
\begin{equation*}
\rho=0.2, \quad
v=0, 
\quad
f=0,
\qquad
e=\begin{cases}1.5, \quad x\in[0.45,0.55],
\\1, \quad \text{otherwise.}\end{cases}
\end{equation*}
The parameters of the model are $\kappa=2,~\sigma=1,~\e=10^{-3},~C_p=10^{-3}$ and $\eta=2$. 
The numerical solution is computed with $\Delta x=10^{-2}$ and compared to Figure~\ref{EM1} with a reference solution obtained by solving (\ref{EM1asy}).
\begin{figure}[htb]
  \centering 
Density\\
\scalebox{0.4}{\rotatebox{270}{\includegraphics{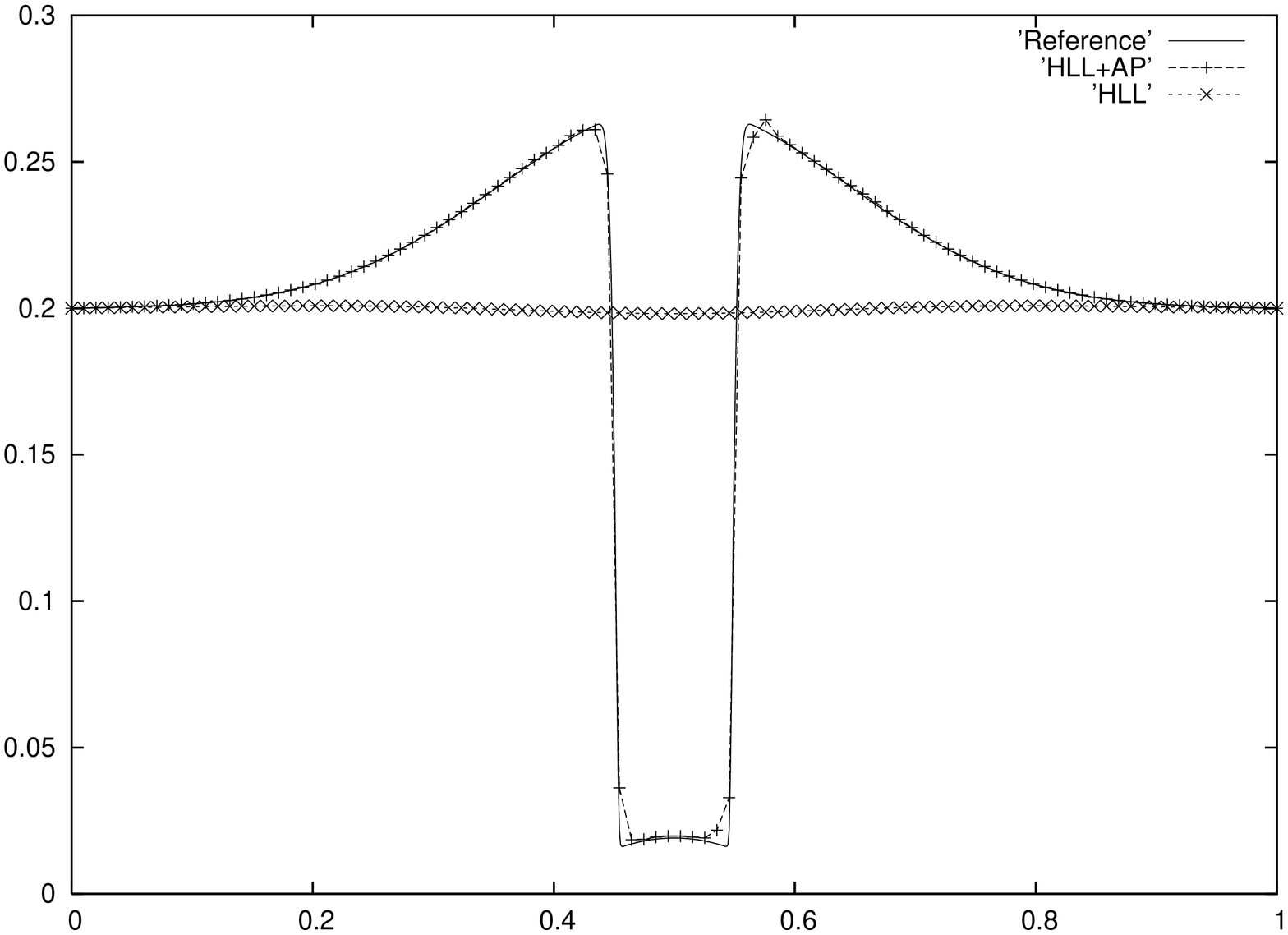}}}\\
Radiative energy
\scalebox{0.4}{\rotatebox{270}{\includegraphics{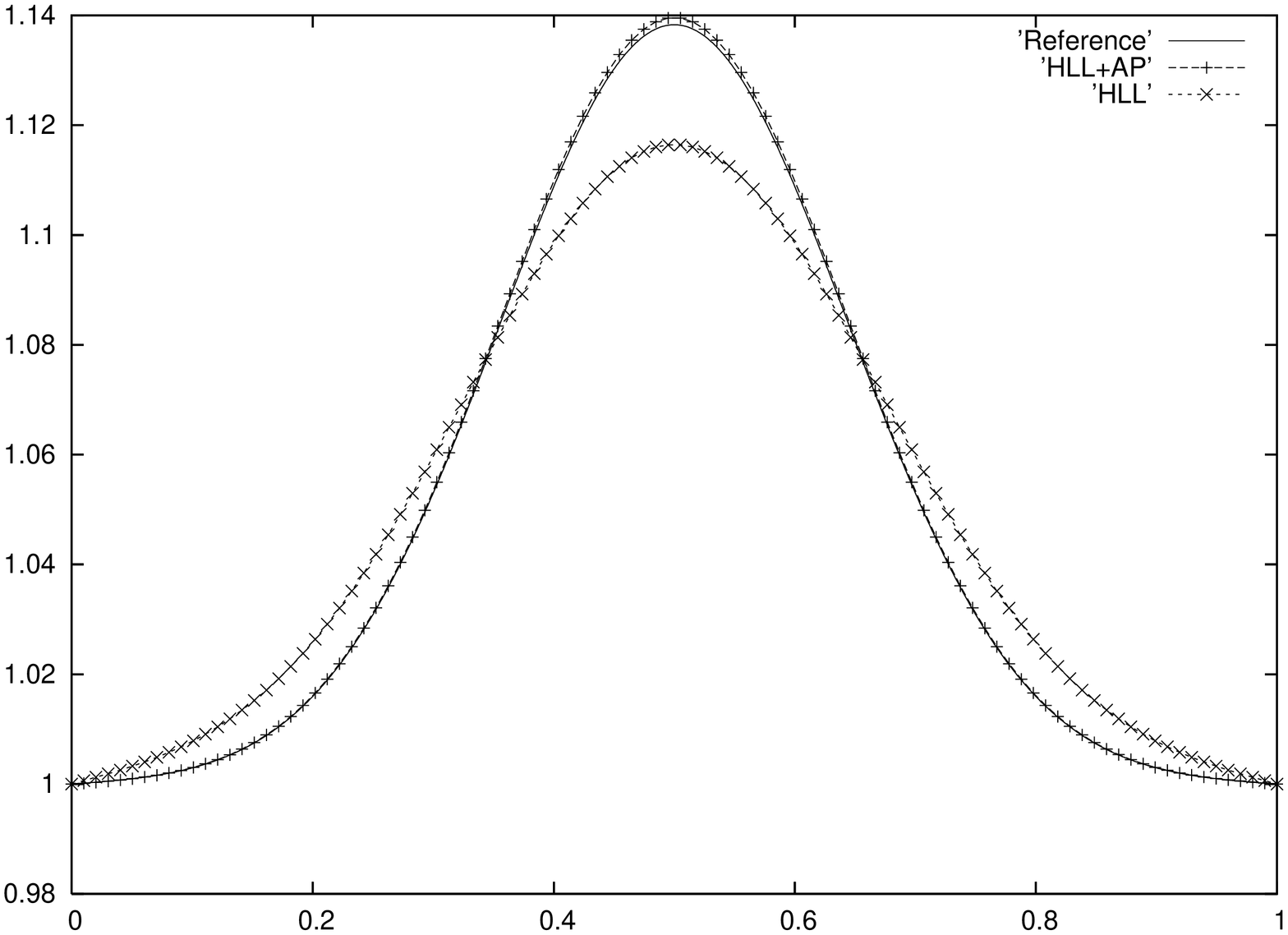}}}\\
\caption{\label{EM1}Reference (full line) and proposed scheme (dashed lines) solution with and without 
asymptotic preserving correction. Density $\rho$ is shown above and Radiative energy $e$ is shown below.}
\end{figure}
Once again, the solution is perfectly captured even on a coarse grid. 
We emphasize that this case is very challenging because of the very specific form of diffusion involved on the density. 


\section{Concluding remarks} 

We have presented a general framerwork to investigate the late-time/stiff-rela\-xation limit of 
a large class of hyperbolic systems. This framework was shown to cover the examples of interest arising in the modeling of 
complex fluid flows when several time-scales are involved. A new class of schemes was proposed for their 
approximation, and we demonstrated that the proposed modification  
was crucial in order to ensure the correct asymptotic behavior for late times.
It would be interesting to make comparisons between the numerical results obtained here and 
concrete experimental results, especially on radiative transfer problems. In future work, we plan to generalize our continuous and discrete 
frameworks to problems with several space dimensions, and 
to numerically investigate the robustness and accuracy of
such multi-dimensional finite volume discretizations in realistic physical situations. 


\section*{Acknowledgments}

The authors gratefully thank Eric Paturel and Friedrich Wagemann for fruitful discussions.
The second author (PLF) was partially supported by  
the Agence Nationale de la Recherche (ANR) through the grant 06-2-134423, and 
by the Centre National de la Recherche Scientifique (CNRS).


\end{document}